\documentclass{article}
\usepackage[utf8]{inputenc}
\usepackage{mathtools}
\usepackage{amsthm}
\usepackage{xcolor}
\usepackage[top=1in, bottom=1.5in, left=1.2in, right=1.2in]{geometry}
\usepackage{amsmath}
\usepackage{amsfonts}
\usepackage{tikz}
\usepackage{pgfplots}
\usepackage{multirow}
\usepackage{algorithm}
\pgfplotsset{compat=newest}

\newtheorem{cor}{Corollary}
\newtheorem{lem}{Lemma}
\newtheorem{prop}{Proposition}
\newtheorem{rem}{Remark}
\newtheorem{thm}{Theorem}

\newcommand{\umax}{\mu_{1}}
\newcommand{\umin}{\mu_{\sizeA}}

\newcommand{\aw}{\textcolor{black}}
\newcommand{\jt}{\textcolor{black}}
\usepackage[normalem]{ulem}

\newcommand{\sizeA}{N}
\newcommand{\resource}{\eta}
\newcommand{\phimat}{\Phi}
\newcommand{\psimat}{\Psi}
\newcommand{\lamalt}{\zeta}
\theoremstyle{definition}

\newcommand{\Dalt}{\Theta}
\newcommand{\poly}{\Omega}
\newcommand{\Lap}{L}

\begin{document}

	\title{Block {Alpha-}Circulant Preconditioners for All-at-Once Diffusion-Based 
		Covariance Operators}

	\author{Jemima M. Tabeart\footnotemark[1] \and Selime G{\"u}rol\footnotemark[2] \and John W. Pearson\footnotemark[3] \and Anthony T. Weaver\footnotemark[2]}

	\date{}
	
	\maketitle
	
	\renewcommand{\thefootnote}{\fnsymbol{footnote}}
	\footnotetext[1]{Department of Mathematics and Computer Science, Eindhoven University of Technology, De Zaale, 5612 AZ, Eindhoven, The Netherlands (\texttt{j.m.tabeart@tue.nl})}
	\footnotetext[2]{CERFACS, 42 Avenue Gaspard Coriolis, 31057 Toulouse Cedex 1, France}
	\footnotetext[3]{School of Mathematics, The University of Edinburgh, James Clerk Maxwell Building, The King’s Buildings, Peter Guthrie Tait Road, Edinburgh, EH9 3FD, United Kingdom }
	
	\begin{abstract}
		{Covariance matrices are central to data assimilation and inverse methods derived from statistical estimation theory. Previous work has considered the application of an all-at-once diffusion-based representation of a covariance matrix operator in order to exploit inherent \jt{parallelism}  in the underlying problem. In this paper, we provide practical methods to apply block $\alpha$-circulant preconditioners to the all-at-once system for the case where the main diffusion operation matrix cannot be readily diagonalized using a discrete Fourier transform. Our new framework applies the block $\alpha$-circulant preconditioner approximately by solving an inner block diagonal problem via a choice of inner iterative approaches.  Our first method applies Chebyshev semi-iteration to a symmetric positive definite matrix, shifted by a complex scaling of the identity. 
			We extend theoretical results for Chebyshev semi-iteration in the symmetric positive definite setting, to obtain computable bounds on the asymptotic convergence factor for each of the complex sub-problems. The second approach transforms the complex sub-problem into a (generalized) saddle point system with real coefficients.     
			Numerical experiments reveal that in the case of unlimited computational resources, both methods can match the iteration counts of the `best-case' block $\alpha$-circulant preconditioner. We also provide a practical adaptation to the nested Chebyshev approach, which improves performance in the case of a limited computational budget. Using an appropriate choice of $\alpha$ our new approaches are robust and efficient in terms of outer iterations and matrix--vector products. } 
	\end{abstract}
	\section{Introduction}
	{Diffusion-based operators can be used to apply a covariance matrix to a vector, thus circumventing the need to explicitly specify a covariance matrix. This makes them particularly convenient for evaluating large covariance matrix--vector products such as those typically required by variational data assimilation algorithms in meteorology and oceanography \cite{weaver_correlation_2016,guillet2019modelling}. Closely related methods have been proposed for solving stationary geophysical inverse problems \cite{buithanh13,tarantola05} and for spatial interpolation in geostatistics \cite{lindgren11,Simpson_2012}. When the diffusion operator is solved implicitly, it approximates a covariance operator for which the kernel is a covariance function from the Whittle--Mat\'ern family \cite{Guttorp_2006,Weaver_2013}.}

	The implicit diffusion-based covariance operator can be defined as a sequence of $\ell$ linear systems, where the solution of one system is used to define the right-hand side of the next system in the sequence. Each linear system involves the same symmetric positive definite (SPD) matrix $A \in \mathbb{R}^{\sizeA\times \sizeA}$, which is constructed from the diffusion operator. The reader is referred to \cite[Section 2]{weaver_correlation_2016} and \cite{weaver2018time} for details. In \cite{weaver2018time}, the authors reformulate the $\ell$-step sequential problem as an ``all-at-once'' problem of dimension $\ell \sizeA \times \ell \sizeA$. For covariance operators in data assimilation, the number of blocks, $\ell$, is usually taken to be even and small (e.g. $\ell\approx 10$
	\cite{weaver2018time, Goux_2024, Chrust_2025}), while the dimension of each block, $\sizeA$, is typically at least several orders of magnitude larger (e.g. $N$ is of the order of $10^6$ for the 2D diffusion operator used in \cite{Chrust_2025}). This is equivalent to solving a system of the form $\mathcal{A}x = b,$ 
	where
	\begin{equation}\label{eq:A}
		\mathcal{A} = \begin{pmatrix}
			{A} \\
			-I_\sizeA & {A} \\
			& \ddots & \ddots \\
			&& -I_\sizeA & {A}
		\end{pmatrix}, \quad x = \begin{pmatrix}
			x_1 \\ x_2 \\ \vdots \\ x_\ell
		\end{pmatrix}, \quad b = \begin{pmatrix}
			b_1 \\ 0\\ \vdots \\ 0
		\end{pmatrix}.
	\end{equation}
	Without loss of generality, we follow the convention of taking $\ell > 2$ to be even within this work (the case $\ell = 2$ is covered by individual portions of our analysis).

	The advantage of solving this system instead of the original sequential system is that matrix--vector products involving the block element $A$ can be applied in parallel when solving the system with an iterative method. However, $\mathcal{A}$ also possesses some challenging structure; it 
	is a {non-diagonalizable} block Toeplitz matrix, and the algebraic expression for $\mathcal{A}^{-1}$ is a lower triangular block matrix involving powers of $A^{-1}$. In this paper, we consider preconditioners that can improve the properties of the preconditioned system compared to the original matrix \eqref{eq:A}.
	
	Prior work in \cite{weaver2018time} tested a simple and inexpensive preconditioner based on replacing $A$ by its diagonal in \eqref{eq:A}. Although this preconditioner can somewhat accelerate convergence, its application, via direct multiplication by a lower triangular block matrix, is not parallelizable across the blocks. 
	A more effective preconditioner is given by the Strang circulant approximation of \eqref{eq:A}, which results in a preconditioned matrix that is diagonalizable and has clustered eigenvalues.
	A parallelizable implementation of the circulant preconditioner for a general parallel-in-time system was proposed in \cite{mcdonald2018preconditioning}, where transforming to Fourier space \aw{in time} allows the preconditioner to be applied blockwise. This work was subsequently extended by \cite{LinNgalphacirculant2021} to consider a block $\alpha$-circulant preconditioner for application to evolutionary partial \jt{differential} equations (PDEs). This, and other similar approaches (see e.g. \cite{bertaccini2003block,hon_sine_2024,li_preconditioned_2024,liu2020fast}), can be applied directly to \eqref{eq:A}. However, as explained later, the full approach proposed in these references is challenging to implement when examining the applications of interest in this study.

	In this paper, we consider how a block $\alpha$-circulant preconditioner 
	can be applied approximately and efficiently, taking into account constraints that would be present in a realistic data assimilation framework such as the one described by \cite{weaver2018time}. The $\alpha$-circulant preconditioner for the problem under consideration is given by 
	\begin{equation}\label{eq:exactPBE}
		\mathcal{P}_{\alpha} = \begin{pmatrix}
			{A} && &-{ \alpha}I_\sizeA \\
			-I_\sizeA & {A} \\
			& \ddots & \ddots \\
			&& -I_\sizeA & {A}
		\end{pmatrix} = I_{\ell} \otimes {A} + C_{\alpha} \otimes(-I_\sizeA)
		~~ \textrm{and}
		~~ C_{\alpha} = 
		\begin{pmatrix}
			0 & & & \alpha \\
			1 & 0 & & \\
			& \ddots & \ddots & \\
			& & 1 & 0
		\end{pmatrix},
	\end{equation}
	where $ C_{\alpha} \in \mathbb{R}^{\ell \times \ell}$ is an {$\alpha$-}circulant matrix and $\otimes$ denotes the Kronecker product.

	\aw{Much previous research on using $\alpha$-circulant matrices as preconditioners assumes that the main blocks (i.e. ${A}$ in this setting) can be diagonalized using a complete basis expansion.}
	For many problems, this is feasible, e.g.~by deploying a Fast Fourier Transform (FFT) or exploiting regularity of grids in the PDE setting. 
	Our new approach is designed to \jt{also} be compatible with an existing implementation of \eqref{eq:A} for the \jt{global} ocean data assimilation system described in \cite{weaver2018time}. 
	\jt{In this setting, ${A}$ cannot be represented in terms of convenient basis functions, such as the spherical harmonics, due to the irregular geometry of the domain, in particular the presence of land masses. Moreover, the large dimension of ${A}$ precludes the use of an explicit, numerically computed eigenvalue--eigenvector expansion.}
	\jt{Therefore, in this work, we develop novel approaches to apply  $\mathcal{P}_{\alpha}^{-1}$ \textit{approximately}, motivated by problems in which a basis transform (such as the FFT) for ${A}$ is not available.}
	In \cite{weaver_correlation_2016}, $\mathcal{A}x = b$ is solved using Chebyshev semi-iteration. The reader is referred to \cite{weaver2018time,weaver_correlation_2016} for a comprehensive discussion of the benefits and suitability of Chebyshev semi-iteration for this application.

	In this paper, we propose new approximations of the preconditioner $\mathcal{P}_{\alpha}$ which avoid applying {{spectral transforms}} in the spatial domain while remaining  computationally affordable.  
	In Section \ref{sec:Background} we apply theory from  \cite{liu2020fast, mcdonald2018preconditioning} to \eqref{eq:A} and show how the majority of the computational work of applying $\mathcal{P}_{\alpha}^{-1}$ reduces to applying the inverse of a block diagonal matrix.  The remainder of the paper proposes two contrasting approaches to approximate the action of $\mathcal{P}_{\alpha}^{-1}$, by solving each of the block inversion problems approximately using an inner iterative approach.  In Section \ref{sec:ApproxPrecond}, we apply Chebyshev semi-iteration (CI) to each of the sub-problems, and study convergence of CI applied to a SPD system shifted by a complex multiple of the identity. In Section \ref{sec:SP}, we reformulate the complex block problem to a real-valued (generalized) saddle point system and propose new preconditioners for this inner problem. 
	We consider the theoretical aspects and practical implementational concerns for both of these approaches.  In Section \ref{sec:Numerics}, we investigate the performance of our  preconditioners for a diffusion problem based on a shifted (negative) Laplacian for different choices of $\alpha$ and in the cases of {{both abundant and limited}} computational resources. Finally, we present our conclusions in Section \ref{sec:Conclusions}.

	\section{{Block $\alpha$-}circulant preconditioning}\label{sec:Background}

	In the case that CI is applied to the preconditioned linear system
	\begin{equation}
		\label{eq:precondsystem}
		\mathcal{P}_{\alpha}^{-1}  \mathcal{A} x = \mathcal{P}_{\alpha}^{-1} b, 
	\end{equation}
	we require estimates of the extreme eigenvalues of the preconditioned matrix $\mathcal{P}_{\alpha}^{-1}  \mathcal{A}$. 
	In Section \ref{sec:EvalsofPrecondSys}, we describe the spectral properties of $\mathcal{P}_\alpha^{-1}\mathcal{A}$. In Section \ref{sec:Kroneckerproductform} we write the Kronecker product form of $\mathcal{P}_\alpha$, such that we may apply an FFT across the $\ell$ blocks. 
	The remainder of the paper describes new methods to approximate the solution of the resulting block-diagonal problem efficiently via nested iterative methods.

	\subsection{Spectral properties of the preconditioned system}\label{sec:EvalsofPrecondSys}
	
	We summarize a number of known spectral results for the block $\alpha$-circulant preconditioner 
	and apply them to our specific problem of interest, \eqref{eq:A}. We begin by proving the invertibility of $\mathcal{P}_{\alpha}$.
	\begin{lem}
		\label{lemma:Zalpha}
		Let $A$ be a symmetric and positive definite matrix of which eigenvalues are given by $\mu_1 \ge \ldots \ge \mu_j \ge \ldots \ge \mu_\sizeA$. Let us assume that $\alpha > 0$ and $\alpha \ne \mu_j^\ell$, $j = 1,\dots,\sizeA$. Then, the matrix defined by 
		$$
		Z_\alpha = \alpha^{-1}(I_\sizeA-\alpha A^{-\ell}),
		$$
		is an invertible matrix. In addition, $\mathcal{P}_{\alpha}$ is an invertible matrix with 
		\begin{align*}
			\mathcal{P}_\alpha^{-1} = \mathcal{A}^{-1} + \mathcal{A}^{-1}E_1Z_\alpha^{-1}E_\ell^{\top}\mathcal{A}^{-1},
		\end{align*}
		where $E_i = e_i\otimes I_N$ and $e_i$ is the $i$-th column of $I_{\ell}$.
	\end{lem}
	\begin{proof}
		By the definition of $\alpha$, $Z_{\alpha}$ does not have a zero eigenvalue, which shows that $Z_{\alpha}$ is invertible. 
		Due to the fact that $\mathcal{P}_\alpha$ can be rewritten as $ \mathcal{P}_\alpha = \mathcal{A}- \alpha E_1 E_\ell^{\top}$, 
		by using the Sherman--Morrison--Woodbury formula we obtain
		\begin{align*}
			\mathcal{P}_\alpha^{-1} & =  \mathcal{A}^{-1} + \alpha\mathcal{A}^{-1}E_1(I_\sizeA - \alpha E_\ell^{\top}\mathcal{A}^{-1}E_1)^{-1}E_\ell^{\top}\mathcal{A}^{-1} \\
			& =  \mathcal{A}^{-1} + \mathcal{A}^{-1}E_1Z_\alpha^{-1}E_\ell^{\top}\mathcal{A}^{-1}.
		\end{align*}
		{Here, we use the relation $ E_\ell^{\top}\mathcal{A}^{-1}E_1 = A^{-\ell}$ which can be proved by induction.}
	\end{proof}

	By using Lemma \ref{lemma:Zalpha} and applying the result of  \cite[Theorem 3.2]{LinNgalphacirculant2021} to \eqref{eq:A} and \eqref{eq:exactPBE}, we obtain the following results on the eigenvalues of the preconditioned system.
	
	\begin{thm}\label{thm:evals}
		Let $\mathcal{A} \in \mathbb{R}^{\ell N \times \ell N}$ and $\mathcal{P}_{\alpha} \in \mathbb{R}^{\ell N \times \ell N}$ be defined by \eqref{eq:A} and \eqref{eq:exactPBE} respectively, where $\alpha$ and ${A} \in \mathbb{R}^{N \times N}$ are defined in Lemma \ref{lemma:Zalpha} \jt{(in particular, $\alpha \ne \mu_j^\ell$, $j = 1,\dots,\sizeA$)}. 
		Then the eigenvalues of $\mathcal{P}_{\alpha}^{-1} \mathcal{A}$ are given by $1$ (with multiplicity $(\ell-1) \sizeA$) and 
		\begin{equation*}
			\frac{\mu_j^{\ell}}{\mu_j^{{\ell}} - \alpha}, \quad j = 1, \dots \sizeA.
		\end{equation*}
	\end{thm}
	
	\begin{proof}
		\jt{Applying the result of \cite[Theorem 3.2]{LinNgalphacirculant2021} to \eqref{eq:precondsystem},} the preconditioned system, $ \mathcal{P}_\alpha^{-1}\mathcal{A} = I +\mathcal{A}^{-1}E_1Z_\alpha^{-1}E_\ell^{\top} $, is of the form
		$$\mathcal{P}_\alpha^{-1}\mathcal{A} = \begin{pmatrix}
			I_\sizeA &&&& A^{-1}{Z}_\alpha^{-1}\\
			& I_\sizeA &&& A^{-2}{Z}_\alpha^{-1}\\
			&&\ddots &&\vdots\\
			&&& I_\sizeA & A^{-(\ell-1)}{Z}_\alpha^{-1}\\
			&&&& I_\sizeA+A^{-\ell}{Z}_\alpha^{-1}\\
		\end{pmatrix}.$$ Hence
		$\mathcal{P}_{\alpha}^{-1}\mathcal{A}$ has an eigenvalue at $1$ with multiplicity $(\ell-1) \sizeA$. The remaining eigenvalues are given by those of $I_\sizeA+\alpha{A}^{-\ell} (I_\sizeA-\alpha{A}^{-\ell})^{-1}$, i.e.
		\begin{equation*}
			1 +\alpha  \mu_j^{-\ell}\left(1 - \alpha \mu_j^{-\ell} \right)^{-1} = 1 +\frac{\alpha}{\mu_j^{{\ell}} - \alpha} = \frac{\mu_j^{{\ell}}}{\mu_j^{\ell} - \alpha},
		\end{equation*}
		from which the required result follows.    
	\end{proof}

	\begin{rem}
		\jt{We note that \cite[Theorem 3.2]{LinNgalphacirculant2021} also provides a bound on the eigenvalues of $\mathcal{P}_{\alpha}^{-1}\mathcal{A}$. Specifically, for any $\eta \in (0,1)$, take $\alpha \in (0,\eta]$. Then $\max_{\lambda \in \sigma(\mathcal{P}_\alpha^{-1}\mathcal{A})}|\lambda-1|\le \frac{\alpha}{1-\eta}$. 
			As the computational performance of CI is strongly dependent on the quality of the estimates of the extreme eigenvalues of the preconditioned system,  explicit spectral information is highly beneficial in this context.} 
	\end{rem}
	
	We now show that although the original system $\mathcal{A}$ is non-diagonalizable \cite{weaver2018time}, choosing ${\alpha\ne \mu_j^\ell}$, $\forall j$, guarantees diagonalizability of the preconditioned system.
	\begin{thm}
		Let $\mathcal{A} \in \mathbb{R}^{\ell N \times \ell N}$ and $\mathcal{P}_{\alpha} \in \mathbb{R}^{\ell N \times \ell N}$ be defined by \eqref{eq:A} and \eqref{eq:exactPBE} respectively, where $\alpha$ and ${A} \in \mathbb{R}^{N \times N}$ are defined in Lemma \ref{lemma:Zalpha}.
		Then  $\mathcal{P}_{\alpha}^{-1}\mathcal{A}$ is diagonalizable.
	\end{thm}
	\begin{proof}
		We show the result by providing an explicit diagonalization. 
		Applying the result of \cite[Lemma 3.3]{LinNgalphacirculant2021} we show that we can diagonalize $I_\sizeA + A^{-\ell}Z_\alpha^{-1} = I_\sizeA - \alpha(\alpha I_\sizeA - A^{\ell})^{-1}$. 
		As this matrix is real and symmetric, it is orthogonally diagonalizable, i.e.
		$$I_\sizeA - \alpha(\alpha I_\sizeA - A^{\ell})^{-1} = Q\Dalt Q^{\top},$$ with the diagonal elements of $\Dalt $ given by $\dfrac{\mu_j^{\ell}}{\mu_j^\ell-\alpha}$.
		We note that as $\alpha > 0$,  $1$ is not an eigenvalue of $\Dalt $. 
		
		Following the result of \cite[Theorem 3.4]{LinNgalphacirculant2021}, we use $Q$ and $\Dalt $ to construct the diagonalization of the preconditioned system. We obtain
		$\mathcal{P}_{\alpha}^{-1}\mathcal{A} = \widehat{Q}\widehat{\Dalt }\widehat{Q}^{-1}$ where 
		
		$$\widehat{Q} = \begin{pmatrix} I_\sizeA &&&&{W_1} \\ & I_\sizeA &&& { W_2}\\ &&\ddots && \vdots \\ &&&I_\sizeA & { W_{{\ell}-1} }\\ &&&&-Q
			
		\end{pmatrix},\qquad \widehat{\Dalt } = \begin{pmatrix} I_\sizeA &&&& \\ & I_\sizeA &&& \\ &&\ddots && \\ &&&I_\sizeA \\ &&&&\Dalt 
			
		\end{pmatrix},$$
		with 
		$${W_j} = A^{-j}Z_\alpha^{-1}Q(I_\sizeA-\Dalt )^{-1}, \quad j = 1,\dots,\ell-1.$$
		
		Since $1$ is not an eigenvalue of $\Dalt $, 
		$    W_j$ is well defined. The invertibility of $\widehat{Q}$ is guaranteed by the invertibility of $Q$, yielding $\mathcal{P}_\alpha^{-1}\mathcal{A} = \widehat{Q}\widehat{\Dalt }\widehat{Q}^{-1}$ for any $\alpha\ne \mu_j^{\ell}$, as required.
	\end{proof}
	
	{To solve the preconditioned linear system \eqref{eq:precondsystem} using CI, it is necessary to determine the \textit{extreme eigenvalues of the preconditioned matrix}, $\mathcal{P}_{\alpha}^{-1} \mathcal{A}$. Next, we show that by choosing $\alpha<\mu_N^\ell$, we guarantee that the smallest eigenvalue of $\mathcal{P}_{\alpha}^{-1} \mathcal{A}$ is given by one, and the largest eigenvalue can be easily calculated from $\mu_N$ and $\alpha$.} 
	
	\begin{cor}\label{cor:extremeevals}
		Let the eigenvalues of $A$ be given by \aw{$\mu_{1} \ge \ldots \ge \mu_{\sizeA-1} \ge \mu_\sizeA > 0$} and suppose that $0 < \alpha < \mu_{\sizeA}^\ell$. {Then $\lambda_{\sizeA}(\mathcal{P}_{\alpha}^{-1} \mathcal{A}) = 1 $ and
			$$\lambda_{\max}(\mathcal{P}_{\alpha}^{-1} \mathcal{A}) = \frac{\mu_{\sizeA}^\ell}{\mu_{\sizeA}^\ell-\alpha}.
			$$
		}
	\end{cor}
	\begin{proof}
		{By selecting $\alpha < \umin^\ell$, it is clear that the smallest eigenvalue of 
			$\mathcal{P}_{\alpha}^{-1} \mathcal{A}$ is given by 1.}
		Fixing $\alpha $, the expression $\frac{\overline{\mu}}{\overline{\mu}-\alpha }$ is monotonically decreasing in $\overline{\mu}$ for $\overline{\mu}>\alpha$, with $\lim_{\overline{\mu}\downarrow \alpha} \frac{\overline{\mu}}{\overline{\mu}-\alpha}\rightarrow \infty$ and $\lim_{\overline{\mu} \uparrow \infty} \frac{\overline{\mu}}{\overline{\mu}-\alpha} = 1$. Therefore, if  $\mu^\ell_{\sizeA}>\alpha$, 
		the ordering of the  eigenvalues of 
		$I_N + A^{-\ell}Z_\alpha$ is reversed, i.e.
		$$\frac{\umin^\ell}{\umin^\ell-\alpha}\ge \frac{\mu_{\sizeA-1}^\ell}{\mu_{\sizeA -1}^\ell-\alpha}\ge \ldots \ge \frac{\umax^\ell}{\umax^\ell-\alpha}>1,$$
		with 
		the largest eigenvalue given by $\frac{\umin^\ell}
		{{\umin^\ell}-\alpha}.$
	\end{proof}
	In this paper, we choose $0 < \alpha < \umin^\ell$, which satisfies the assumptions for all the theoretical results in this section. For the case study considered in Section \ref{sec:Numerics}, $\umin>1$ (see Appendix \ref{sec:Append}), meaning we are free to select $\alpha \le 1$.  We investigate different choices of $\alpha$ numerically in Section \ref{sec:Numerics}.

	\subsection{Kronecker product form of the preconditioner}\label{sec:Kroneckerproductform}
	
	A key computational advantage of the block $\alpha$-circulant preconditioner \eqref{eq:exactPBE} is that the dependency between different blocks of $\mathcal{P}_\alpha$ can be decoupled using a discrete Fourier transform (DFT). This permits us to re-write $\mathcal{P}_\alpha^{-1}$ so that the bulk of the computational effort lies in the application of the inverse of a block diagonal matrix to a vector. In this section we discuss this decomposition via the Kronecker structure of $\mathcal{P}_\alpha$.
	
	We exploit the block {$\alpha$-}circulant structure to write $\mathcal{P}_{\alpha}$
	in the following Kronecker product form:
	{
		\begin{align}
			\nonumber
			\mathcal{P}_{\alpha} &= I_{\ell} \otimes {A} + C_{\alpha} \otimes(-I_\sizeA)\\
			\nonumber
			& = I_{\ell} \otimes {A} + \Gamma_{\alpha}^{-1}U\Lambda U^*\Gamma_{\alpha} \otimes(-I_\sizeA)\\
			& = (\Gamma_{\alpha}^{-1}U\otimes I_\sizeA) (I_{\ell}\otimes {A} - \Lambda \otimes I_\sizeA)(U^*\Gamma_{\alpha}\otimes I_\sizeA).
			\label{eq:palpha}
		\end{align}
		Here, $\Gamma_{\alpha} = \textrm{diag}(1,\alpha^{1/\ell},\alpha^{2/\ell},\dots,\alpha^{(\ell-1)/\ell})$ is a diagonal scaling matrix made up of powers of $\alpha$, $\Lambda \in \mathbb{R}^{\ell\times \ell}$ is a diagonal matrix of scaled {$\ell$-th} roots of unity, i.e. 
		\begin{equation*}
			\lambda_j = \alpha^{1/\ell}\exp \left(\frac{2\pi i (j-1)}{\ell}\right), {\quad j=1, \ldots, \ell},
		\end{equation*}
		and the $j$-th column of $U$ is given by 
		$$U_j = \frac{1}{\sqrt{\ell}}\left(1,\exp\left(\frac{2\pi i(j-1)}{\ell}\right), \exp\left(\frac{2\pi i(2(j-1))}{\ell}\right),\dots,\exp\left(\frac{2\pi i(\ell-1)(j-1)}{\ell}\right)\right)^{\top},$$
		i.e. Fourier modes.  In the case where $\alpha=1$, we have $\Gamma_1=I_{\ell}$ and the diagonal entries of $\Lambda$ are {$\ell$-th} roots of unity. If $\alpha\ne1$, we refer to $\lambda_j$ as scaled ({$\ell$-th}) roots of unity.

		The terms in the outer brackets in expression \eqref{eq:palpha} are trivial to invert using the properties of Kronecker products and the definitions of $U$, $\Gamma_{\alpha}$.
		The inner matrix is a block diagonal matrix with $j$-th block ${A}-\lambda_j I_\sizeA$. Therefore, the inverse of $\mathcal{P}_{\alpha}$ is given by
		\begin{equation}\label{eq:Paform}
			\mathcal{P}_{\alpha}^{-1} =(\Gamma_{\alpha}^{-1}U\otimes I_\sizeA) (I_{\ell}\otimes {A} - \Lambda \otimes I_\sizeA)^{-1}(U^*\Gamma_{\alpha}\otimes I_\sizeA).
		\end{equation}
		
		As $A$ and $I_\sizeA$ are simultaneously diagonalizable, recalling that $A$ is SPD and hence admits the decomposition ${A} = X\Phi X^\top$, we can go one step further and write
		\begin{equation*}
			\mathcal{P}_{\alpha} 
			= (\Gamma_{\alpha}^{-1}U\otimes I_\sizeA)(I_{\ell}\otimes X)(I_{\ell}\otimes\Phi  - \Lambda \otimes I_\sizeA)(I_{\ell} \otimes X^\top)(U^*\Gamma_{\alpha}\otimes I_\sizeA).
		\end{equation*}
		
		In this formulation, applying $\mathcal{P}_{\alpha}^{-1}$ requires the inversion of a diagonal matrix $\Gamma_{\alpha}$, which is computationally trivial. 
		In much of the parallel-in-time literature (see e.g. \cite{hon_sine_2024,li_preconditioned_2024,mcdonald2018preconditioning,StollM2014Osoa}), applications are examined for which obtaining a decomposition ${A} = X\Phi X^{\top}$ is feasible, for instance via a spatial fast Fourier transform (FFT).
		However, for some applications, including the motivating diffusion-based covariance problem,  a spatial FFT-based scheme is not always readily available. 
		In the remainder of this work, we therefore consider computationally feasible methods to apply the block $\alpha$-circulant preconditioner approximately for the setting where a DFT or full eigendecomposition of ${A}$ is unavailable. {For this reason, we focus on the formulation of $\mathcal{P}_{\alpha}^{-1}$ expressed by \eqref{eq:Paform}, which does not require the decomposition of $A$. Instead, this requires matrix--vector products with the inverse of $(I_{\ell}\otimes {A} - \Lambda \otimes I_\sizeA)$, which can be approximately applied to a vector by  solving linear systems:}
		\begin{equation}
			\label{eq:full_blockproblem}
			(I_{\ell}\otimes {A} - \Lambda \otimes I_\sizeA)x = b.
		\end{equation}
		{Noting that the matrix $(I_{\ell}\otimes {A} - \Lambda \otimes I_\sizeA)$ is a block diagonal matrix, the solution of the linear system \eqref{eq:full_blockproblem} can hence be obtained by solving each linear system:
			\begin{equation}
				\label{eq:blockproblem}
				({A} - \lambda_j I_\sizeA) x_j = b_j, \quad j = 1,\dots, \ell,
			\end{equation}
			in parallel, where $x_j$ and $b_j$ are $N$-dimensional vectors.}

		{We consider two approaches for the solution of linear systems given by \eqref{eq:blockproblem}}. Our first approach, presented in Section \ref{sec:ApproxPrecond}, applies CI to the (possibly complex) inner problems \eqref{eq:blockproblem}. The second approach, in Section \ref{sec:SP}, reformulates \eqref{eq:blockproblem} as a (generalized) saddle point system taking only real values, and applies MINRES \cite{PaigeSaunders} with an appropriate choice of preconditioner to solve each reformulated inner problem.  For each of the methods we present a theory of convergence for the inner problem. In Section \ref{sec:Numerics} we consider the numerical performance of the overall preconditioners for a realistic setting when each of the sub-problems \eqref{eq:blockproblem} are only solved approximately using a small, fixed number of iterations.

		\section{Applying the preconditioner via nested Chebyshev semi-iteration}\label{sec:ApproxPrecond}
		In this section we consider how to apply the preconditioner \eqref{eq:Paform} approximately by using CI to solve each of the block problems \eqref{eq:blockproblem}. For $A\in\mathbb{R}^{\sizeA \times \sizeA}$ SPD with eigenvalues in $[\mu_{N},\mu_1]$, let us denote 
		\begin{equation}\label{eq:B definition}
			B_j = {A} - \lambda_j I_\sizeA, \quad {j = 1, \ldots, \ell}.
		\end{equation}
		
		As $\ell$ is taken to be even in this work, $\lambda_{\ell/2+1} = -\alpha^{1/\ell}$, and {$\lambda_1 = \alpha^{1/\ell}$}. The remaining $\lambda_j$ have non-zero complex part with $\lambda_j = \overline{\lambda}_{\ell+2-j}$, $j =\{2,3, \ldots, \ell/2 \}$, and $\Re(\lambda_j)<\Re(\lambda_k)$ for $1\le k < j\le \ell/2+1$ or $\ell/2+1\le j<k\le \ell$ (where $\Re(\lambda)$ denotes the real part of $\lambda$). An example of this labelling convention is illustrated for $\ell =10$ in Figure \ref{fig:l=10 example}.
		
		\begin{figure}[ht]
			\centering
			\begin{tikzpicture}
				\pgfmathsetmacro{\n}{10}
				\begin{axis}[axis equal,
					axis lines=center,grid=major,
					xlabel=$\Re$,
					xtick={0},
					extra x tick style={x tick label style={above, yshift=0ex,xshift=2.9ex}},
					extra x ticks={-1,1},
					extra x tick labels={$-\alpha^{1/\ell}$,\hspace{-10.7ex}$\alpha^{1/\ell}$},
					ytick={0},
					extra y tick style={y tick label style={above, yshift=0.1ex,xshift=0.3ex}},
					extra y ticks={-1,1},
					extra y tick labels={\hspace*{-1.4ex}$-\alpha^{1/\ell}i$,  \hspace{0.3ex}\vspace{-10.1ex}$\alpha^{1/\ell}i$},
					xmax=1.5,
					xmin=-1.5,
					ymax=1.7,
					ymin=-1.5,
					ylabel=\hspace{1mm}$\Im$,
					samples=10,
					disabledatascaling]
					\draw[help lines, black] (0,0) circle (1);
					\foreach \t in {1,...,\n} {
						\edef\temp{\noexpand
							\node[fill=blue, circle, draw=blue, scale=0.5] at ( {cos((360*\t)/\n)}, {sin((360*\t)/\n)} ) {};
						}\temp}
					\foreach \t in {1,...,\n} {
						\edef\temp{\noexpand
							\node[fill=white, circle, draw=none, scale=0.7,text=blue] at ( {1.25*cos((360*(\t-1))/\n)}, {1.25*sin((360*(\t-1))/\n)} ) {{$\lambda_{\t}$}};             
						}\temp}
					
				\end{axis}
			\end{tikzpicture}
			\caption{Visualization of scaled $10$-th roots of unity on $\mathbb{C}$-plane, with the labelling convention going anti-clockwise from $\lambda_1 = \alpha^{1/\ell}$.}
			\label{fig:l=10 example}
		\end{figure}
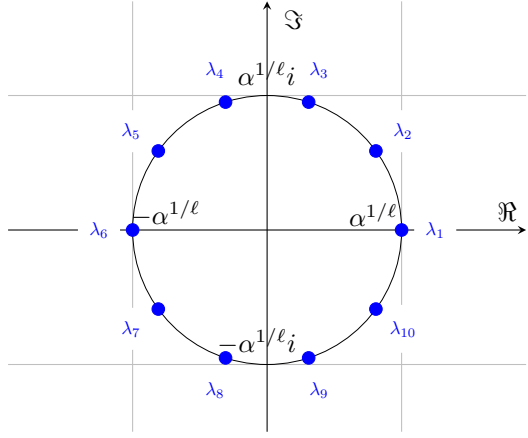
		
		As a result, $B_j \in \mathbb{C}^{\sizeA \times \sizeA}$ has the special structure 
		$$
		B_j = A - (\Re(\lambda_j) + i \;\Im(\lambda_j) )I_N
		$$
		where $\Im(\lambda)$ denotes the imaginary part of $\lambda$ and $i^2 = -1$.
		Since $A$ is SPD, $B_j$ are normal matrices for $j=1, \ldots, \ell$. Hence, it is unitarily diagonalizable, i.e. there exists a unitary matrix $V_j$ such that 
		\begin{equation}
			\label{eq:Bdecomp}
			B_j = V_j \Xi_j V_j^{\ast},
		\end{equation}
		with $\Xi_j = \textrm{diag}(\xi_{j}^{(1)}, \xi_{j}^{(2)}, \ldots, \xi_{j}^{(\sizeA)})$ a diagonal matrix containing the eigenvalues of $B_j$. 
		\begin{lem}\label{lem:evals}
			The eigenvalues of $B_j$, given by $\xi_j^{(k)}$ for $k = 1,\ldots, \sizeA$, $j = 1,\dots,\ell$, have imaginary part given by 
			$\Im(\lambda_j)$ and $\xi_j^{(k)} \in S_j := [\umin - \Re(\lambda_j) - i \; \Im(\lambda_j), \umax - \Re(\lambda_j) - i \;\Im(\lambda_j)]$, where $S_j$ is a line segment parallel to the real axis.\end{lem}
		\begin{proof}
			The eigenvalues of $B_j$ are given by $ \mu_k - \Re(\lambda_j) - i\;\Im(\lambda_j) := \xi_j^{(k)} = \Re(\xi_j^{(k)}) - i\;\Im(\lambda_j) $ for $k = 1,\dots,N$.
			Hence, all eigenvalues $\xi_j^{(k)}$ (for fixed $j$) lie on the line segment $S_j := [\umin - \Re(\lambda_j) - i\;\Im(\lambda_j), \umax - \Re(\lambda_j) - i\;\Im(\lambda_j)]$.
		\end{proof} 
		Note that, for $j' \in \{1, \ell/2 + 1\}$, 
		$\Im(\lambda_{j'}) = 0$. Hence, $B_{j'}$ has real eigenvalues. As $\umin > {\alpha}^{1/\ell}$,  $ B_{j'} (= {A}\pm {\alpha^{1/\ell}}I_\sizeA)$ is symmetric positive definite for both real roots, with the eigenvalues of the perturbed system lying on the positive real axis, i.e. 
		\begin{align*}
			\xi_1^{(k)} & \in S_1 = [\umin - \alpha^{1 / \ell}, \umax- \alpha^{1 / \ell}], \\ 
			\xi_{\ell/2 + 1}^{(k)} & \in S_{\ell/2 + 1} = [\umin + \alpha^{1 / \ell}, \umax + \alpha^{1 / \ell}], 
		\end{align*}
		for $k = 1, \ldots, \sizeA$.

		We next explore how to use CI to solve a linear system involving a normal matrix $B_j$. {Note that for non-normal matrices, Chebyshev semi-iteration would not be guaranteed to converge in generality \cite{manteuffel1978}. Hence, the normality of the matrix $B_j$ is crucial to our conclusions.}
		
		Given an initial guess $\chi_0$ for a true solution $\chi^{\ast}$, a polynomial-based iterative method produces, at iteration $p$, an estimate $\chi_p$ such that the error $e_p = \chi^{\ast} - \chi_p$ can be written as 
		\begin{equation*}
			{e}_p = \poly_p(B_j) {e}_0,
		\end{equation*}
		with $\poly_p(B_j)$ denoting a polynomial of degree at most $p$ such that $\poly_p(0) = 1$.  
		Consequently, assuming non-zero initial error, we have
		\begin{equation}\label{ep_error}
			\frac{\| {e}_p \|_2}{\| {e}_0 \|_2} \le \|  \poly_p(B_j) \|_2.
		\end{equation}

		From the spectral decomposition of $B_j$ given by \eqref{eq:Bdecomp}, we have
		\begin{align}
			\nonumber
			\|  \poly_p(B_j) \|_2 = \|  \poly_p(\Xi_j) \|_2 ={}& \text{spec} ( \poly_p(\Xi_j) ) \\
			\label{eq:norm_bound} ={}& \max_{k = 1, \ldots, \sizeA} |  \poly_p(\mu_k - \Re(\lambda_j) - i\; \Im(\lambda_j))  | \le \max_{{z} \in S_j} |  \poly_p( {z})|,
		\end{align}
		where $\text{spec}(\cdot)$ denotes the spectral radius of a matrix. 
		Each polynomial-based method uses a different criterion to choose the polynomial $\poly_p$.  It is well known \cite[Sec 5.3 and Appendix B]{axelsson1996iterative} that when $S_j$ consists solely of positive real eigenvalues, CI selects $\poly_p$ as the shifted-and-scaled Chebyshev polynomials. These solve the minimax problem \cite[Appendix A3]{weaver_correlation_2016}:
		\begin{equation}
			\label{Prob_minmax}
			\min_{\substack{ \text{deg}(\poly) \le p \\ \poly_p(0)  = 1} }  \; \max_{z \in S_j} |  \poly_p( z)  |.
		\end{equation}
		
		In the case that $S_j\subset \mathbb{R}$, the Chebyshev polynomials are constructed based on the knowledge of the extreme eigenvalues. 
		When the set $S_j$ includes complex eigenvalues, we define an ellipse, $E_j$ such that $S_j\subseteq E_j$ and $0\notin E_j$ \cite{gutknecht2002chebyshev, manteuffel1977tchebychev, gergelits2014composite}. Then the minimax problem is solved over $E_j$ \cite{axelsson1996iterative}, and the Chebyshev polynomials are constructed based on the parameters of the ellipse. The computational performance of CI depends on the choice of $E_j$, which is not uniquely defined. For some special cases of $S_j$, $E_j$ can be selected a priori to ensure good numerical performance of a polynomial-based iterative method. For example, the case where $S_j$ is a line segment of the complex plane which is both parallel to the imaginary axis and symmetric about the real axis is studied in \cite{freund1986class}. For this case the authors provide an explicit solution to the minimax problem \eqref{Prob_minmax}, and show that the optimal polynomials are \textit{not} given by the Chebyshev polynomials. In our problem setting, $S_j$ are line segments parallel to the real axis, so the results of \cite{freund1986class} cannot be applied directly. In Section \ref{sec:ComplexLambda} we derive similar results, showing that the polynomials that solve the minimax problem \eqref{Prob_minmax} are distinct from the Chebyshev polynomials. However, these optimal polynomials do not lead to a practical algorithm. We hence propose an alternative method based on CI, for which we can prove bounds on the asymptotic convergence factor in terms of only the real parts of $S_j$.

		Based on the properties presented in this section, we next study the asymptotic convergence of CI applied to \eqref{eq:blockproblem} for two different cases:
		\begin{itemize}
			\item $\lambda_j\in\mathbb{R}$ (Section \ref{sec:RealLambda}),
			\item $\lambda_j \in \mathbb{C}\setminus \mathbb{R}$ (Section \ref{sec:ComplexLambda}).
		\end{itemize}

		\subsection{Convergence of inner problem with real perturbation}\label{sec:RealLambda}
		
		In this section, we focus on the case where $B_j$ has real eigenvalues, i.e. $j' \in \{1, \ell/2 + 1\}$. In this case, CI using the shifted-and-scaled Chebyshev polynomials solves the minimax problem \eqref{Prob_minmax}. 
		When CI is applied to an SPD matrix, the number of iterations required to reach a desired tolerance, $\epsilon$, can be calculated in advance \cite{axelsson1996iterative} as stated by the following theorem:
		
		\begin{thm}[\cite{axelsson1996iterative}, Section 5.3]\label{thm:AxelssonIteration}
			Let $M$ be symmetric positive definite, with eigenvalues in $[\upsilon_{\min},\upsilon_{\max}]$. The CI method applied to $M\chi = \beta$ converges to a given tolerance, $\epsilon$, in $\lceil p^* \rceil$ iterations, where
			\begin{equation*}
				p^* = \frac{\ln\left(\frac{1}{\epsilon}+ \sqrt{\frac{1}{\epsilon^2}-1}\right)}{\ln\left(\frac{1+\sqrt{\upsilon_{\min}/\upsilon_{\max}}}{1-\sqrt{\upsilon_{\min}/\upsilon_{\max}}}\right)}.
			\end{equation*}
		\end{thm}
		The result of Theorem \ref{thm:AxelssonIteration} can be used to compare the performance of CI applied to $B_1$ and $B_{\ell/2+1}.$ 
		For a fixed choice of $\epsilon$, we can write
		\begin{equation*}
			p^*_1 \; {\ln(1/\sigma_1)} = p^*_{\ell/2+1} \ln(1/\sigma_{\ell/2+1}),
		\end{equation*}
		where
		$\sigma_1 = \frac{\sqrt{\kappa(
				{B_1})}-1}{\sqrt{\kappa({B_1})}+1}$,  
		$\sigma_{\ell/2+1} = \frac{\sqrt{\kappa({B_{\ell/2+1})}}-1}{\sqrt{\kappa({B_{\ell/2+1})}}+1}$, $p^*_1$ and $p^*_{\ell/2+1}$ denote the computed $p^*$ with the relevant eigenvalue shifts, and 
		$\kappa(\cdot)$ denotes the condition number of a SPD matrix. 
		As $\kappa(B_1) > \kappa(B_{\ell/2+1})$, it is clear that $1>\sigma_{1}>\sigma_{\ell/2+1}$ and hence $p^*_1 > p^*_{\ell/2+1}$.
		
		This can be interpreted as follows: to achieve the same tolerance, applying CI to $B_1$ requires roughly a factor of $\frac{\ln(\sigma_{\ell/2+1})}{\ln(\sigma_1)}$ more iterations than applying CI to $B_{\ell/2+1}$. Alternatively, applying the same number of iterations of CI to both problems results in a smaller tolerance on the solution of $B_{\ell/2+1}$ than on the solution of $B_1$. We exploit this relation in Section \ref{sec:Numerics} to accelerate the convergence when using the nested CI preconditioner in the case of a limited computational budget.
		
		\subsection{Convergence of inner problem with complex perturbation}\label{sec:ComplexLambda}
		We now show that CI can also be applied to \eqref{eq:blockproblem} for the complex scaled roots of unity, $\lambda_j$, $j\in[2,\ell/2]\cup[\ell/2+2,\ell]$, with guaranteed upper bounds on the asymptotic convergence factor. In this setting $B_j$ is no longer symmetric positive definite, so the theoretical results from Section \ref{sec:RealLambda} do not hold. However, we can exploit the particular spectral structure of $B_j$ to obtain similar results for complex values of~$\lambda_j$ by selecting the limit case $E_j = S_j$. We now show that CI is guaranteed to converge when applied to $B_j$, and compute an upper bound on the error at each iteration.

		\begin{prop}\label{prop:errorbound}
			Let $B_j$ be defined as in \eqref{eq:B definition}. 
			Then CI converges when applied to $B_j$, with  the error at the $p$-th iteration being bounded above by \begin{equation*}
				\frac{\| {e}_p \|_2}{\| {e}_0 \|_2} \le {\left| T_p \left(\frac{\xi_j^{(1)}+\xi_j^{(N)}}{\xi_j^{(1)}-\xi_j^{(N)}}\right) \right|}^{-1}.\end{equation*}
		\end{prop}
		
		\begin{proof}
			The CI method takes $\poly_p$ in \eqref{ep_error} to be the shifted-and-scaled Chebyshev polynomial, i.e.
			\begin{equation*}
				{\poly_p(\mu) = \frac{T_p \left( \frac{\xi_j^{(1)} + \xi_j^{(N)} - 2 \mu}{\xi_j^{(1)} - \xi_j^{(N)}} \right)}{T_p \left( \frac{\xi_j^{(1)} + \xi_j^{(N)}}{\xi_j^{(1)} - \xi_j^{(N)}} \right)} =}  \frac{T_p \left(\frac{\umax+\umin-2\lambda_j-2\mu}{\umax-\umin}\right)}{T_p \left(\frac{\umax+\umin-2\lambda_j}{\umax-\umin}\right)},
			\end{equation*}
			where $\xi_j^{(1)}$ and $\xi_j^{(N)}$ are the eigenvalues of $B_j$ with largest and smallest real part, respectively, and $\mu \in \mathbb{C}$.
			Hence,
			\begin{align*}
				\max_{\mu \in [{\xi_j^{(\sizeA)},\xi_j^{(1)}}]} |  \poly_p( \mu) | ={}& \max_{\mu \in [{\xi_j^{(\sizeA)},\xi_j^{(1)}}]} \frac{\left| T_p \left(\frac{\umax+\umin-2\lambda_j-2\mu}{\umax-\umin}\right) \right|}{\left| T_p \left(\frac{\umax+\umin-2\lambda_j}{\umax-\umin}\right) \right|} \\
				={}& \frac{1}{\left|T_p \left( \frac{\xi_j^{(1)} + \xi_j^{(N)}}{\xi_j^{(1)} - \xi_j^{(N)}}  \right)\right|},
			\end{align*}
			by setting the argument {$\mu = \umin- \lambda_j$} to obtain the last line. {Using this relation, together with \eqref{ep_error} and \eqref{eq:norm_bound},
				we obtain the desired result.}
		\end{proof}
		We note that Proposition \ref{prop:errorbound} applies to both complex and real values of $\lambda_j$. For $\lambda_{j'}$, $j'\in \{1,\ell/2+1\}$, the Chebyshev polynomial, $T_p$, is evaluated at a real value as $\xi_{j'}^{(1)}+\xi_{j'}^{(N)}\in\mathbb{R}$, whereas for $j \in [2,\ell/2]\cup[\ell/2+2,\ell]$, it is evaluated at a complex value as $\xi_{j}^{(1)}+\xi_{j}^{(N)}\in\mathbb{C}$.

		{The next lemma provides an upper bound
			on the asymptotic convergence {factor}, $\rho(\lambda_j)$, of the shifted system $B_j$ when using a complex shift value, $\lambda_j$.}

		\begin{lem}[\cite{axelsson1996iterative}, p.\,187]\label{lem:asymptoticrate}
			The asymptotic convergence factor of CI applied to $B_j$ is bounded via
			\begin{align}
				\nonumber \rho(\lambda_j) \leq \lim_{p \rightarrow \infty} \, \left\{ \max_{\mu \in [\xi_j^{(\sizeA)},\xi_j^{(1)}]} \, 
				\frac{\left| T_p \left(\frac{{ \xi_j^{(1)}+\xi_j^{(N)} -2\mu}}{\xi_j^{(1)}-\xi_j^{(N)}}\right) \right|}{\left| T_p \left(\frac{\xi_j^{(1)}+\xi_j^{(N)}}{\xi_j^{(1)}-\xi_j^{(N)}}\right) \right|} \right\}^{1/p} 
				= \frac{1}{\lim_{p\rightarrow \infty} \left|T_p \left(\frac{\xi_j^{(1)}+\xi_j^{(N)}}{\xi_j^{(1)}-\xi_j^{(N)}}\right)\right|^{1/p}}.
			\end{align}
			
		\end{lem}
		
		In the next result, we compute a slightly weaker bound that only uses the real part of~$\lambda_j$. 
		
		\begin{thm}\label{thm:convergence}
			Let $B_j$ be defined as in \eqref{eq:B definition}. We denote by  $\rho(\lambda_j)$ the
			{asymptotic convergence factor at the point $\lambda_j$ for  CI} applied to $B_j$ with endpoints $ [\xi_{j}^{(N)},\xi_{j}^{(1)}]$. This {asymptotic convergence factor for the complex value $\lambda_j$} can be bounded by the {asymptotic convergence factor for its real part $\Re(\lambda_j)$}, i.e. 
			\begin{equation}\label{eq:real rate of conv}
				\rho(\lambda_j) < \rho(\Re(\lambda_j)) = {\sigma_j},
			\end{equation}
			{with $\sigma_j = \frac{\sqrt{\kappa(\Re(B_j))}-1}{\sqrt{\kappa(\Re(B_j))}+1}$}.
		\end{thm}
		\begin{proof}
			Following Lemma \ref{lem:asymptoticrate}, it suffices to show that
			$$
			\lim_{p\rightarrow \infty} \left|T_p \left(\frac{\xi_j^{(1)}+\xi_j^{(N)}}{\xi_j^{(1)}-\xi_j^{(N)}}\right)\right|^{1/p}  > \lim_{p\rightarrow \infty} \left|T_p \left(\frac{\Re(\xi_j^{(1)}+\xi_j^{(N)})}{\xi_j^{(1)}-\xi_j^{(N)}}\right)\right|^{1/p}.
			$$

			{For a complex-valued $x$, 
				the Chebyshev polynomial $T_p$ can be defined as~\cite{gutknecht2002chebyshev}:
				$$
				T_p(x) = \frac{1}{2} \left [ \left( x + \sqrt{x^2 -1} \right)^p + \left( x - \sqrt{x^2 -1} \right)^p \right].
				$$
				By applying the convention of taking the square root with positive real part~\cite[Section 1.1]{bak2010complex} for $x=\frac{\xi_j^{(1)}+\xi_j^{(N)}}{\xi_j^{(1)}-\xi_j^{(N)}} = a+ ib $ with $a > 0$,
				we obtain
				$$\left | x + \sqrt{x^2 -1} \right | > |x| > 1,$$
				which leads to
				$$\lim_{p\rightarrow\infty} \left|T_p(x)\right|^{1/p} = \left|x+\sqrt{x^2-1}\right|.$$ Using this result, it remains to show that
				\begin{equation}
					\label{Eq:inequality_complex_real}
					\left|x  + \sqrt{x^2-1}\right|  >  \left|\Re(x)  + \sqrt{\Re(x)^2-1}\right| = \left|a  + \sqrt{a^2-1}\right|. 
				\end{equation}
				Note that
				\begin{align*}
					\left | x + \sqrt{x^2 -1} \right |^2 & > \left( \Re \left( x + \sqrt{x^2 -1} \right) \right)^2 \\
					& = \left( a + \sqrt{\frac{\sqrt{(a^2 - b^2 - 1)^2 + 4a^2b^2} + (a^2-b^2-1)}{2} } \right)^2,
				\end{align*}
				where the square root is expressed in terms of its positive real part. Therefore to prove the bound~\eqref{Eq:inequality_complex_real} it suffices to show that
				\begin{equation*}
					\sqrt{(a^2 - b^2 - 1)^2 + 4a^2b^2} + (a^2-b^2-1)> 2 (a^2 - 1)
				\end{equation*}
				since $|a| > 1$. This inequality simplifies to 
				$$
				\sqrt{(a^2 + b^2 - 1)^2 + 4a^2b^2} > a^2 + b^2 - 1,
				$$
				which clearly holds, hence proving~\eqref{Eq:inequality_complex_real}. Therefore, $\rho(\lambda_j) < \rho(\Re(\lambda_j))$ and we can use  \cite[Section 5.3]{axelsson1996iterative} to obtain 
				$$\rho(\lambda_j) < \rho(\Re(\lambda_j)) = \sigma_j.$$
			}
		\end{proof}

		We can draw a number of conclusions from the bound \eqref{eq:real rate of conv}:
		\begin{itemize}
			\item A key advantage of the result of Theorem \ref{thm:convergence} compared to the result of Lemma \ref{lem:asymptoticrate} is that \eqref{eq:real rate of conv} is readily computable, only requiring knowledge of the extreme eigenvalues of $A$, and $\Re(\lambda_j)$.
			\item The upper bound for $\rho(\lambda_j)$, specifically $\sigma_j$, has its analogous term appearing in Theorem~\ref{thm:AxelssonIteration}. Motivated by this observation, in Section \ref{sec:Numerics} we use \eqref{eq:real rate of conv} to design a resource-allocation heuristic based on the approach discussed after Theorem \ref{thm:AxelssonIteration}.
			\item As $\ell$ is taken to be even, scaled roots of unity appear as complex conjugate pairs. The bound on the asymptotic convergence factor \eqref{eq:real rate of conv} gives the same value for complex conjugates $\lambda_j$ and $\overline{\lambda}_j$.
			\item For two scaled roots of unity $\lambda_m$, $\lambda_n$ with $\Re(\lambda_m) < \Re(\lambda_n)$, we have 
			{$\sigma_m  <\sigma_n$}.
			We note that this result holds for the upper bounds on the asymptotic convergence factor rather than the rates of convergence themselves. Similarly, for complex $\lambda_j$, the upper bound for the asymptotic convergence factor is bounded below by the asymptotic convergence factor for ${\lambda_{\ell/2+1}} = -{\alpha^{1/\ell}}$ (see Figure \ref{fig:l=10 example}). 
			\item  The bound in Theorem \ref{thm:convergence} becomes sharper as $\Im(\lambda_j)$ approaches $0$. In particular, the bound is tighter for scaled roots of unity when $\alpha \ll 1$ compared to $\alpha = 1$. In Section \ref{sec:Numerics}, we compare the bound and numerical convergence for different choices of $\alpha$.
		\end{itemize}

		\subsection{{Optimal polynomial for the minimax problem}}
		Unlike $\lambda_1,\lambda_{\ell/2+1}\in\mathbb{R}$, for complex scaled roots of unity the Chebyshev polynomials no longer solve the minimax problem \eqref{Prob_minmax}. 
		Indeed, for the case $p=1$, we  prove that the maximal value obtained for the Chebyshev polynomial on the desired range is strictly larger than for the optimal polynomial, $\poly^*_1$.
		
		\begin{prop}
			Let $B_j$ be given by \eqref{eq:B definition}, \jt{and exclude the redundant case $\Re(\xi_j^{(1)}) = \Re(\xi_j^{(N)})$}. For $p=1$ the maximal value of the optimal polynomial $\poly^*_1$ is strictly smaller than the maximum value of the scaled Chebyshev polynomial, i.e.
			$$
			\max_{\mu \in [\xi_j^{(\sizeA)},\xi_j^{(1)}]}| \poly_1^{\ast}(\mu) |\le {\left| T_p \left(\frac{\xi_j^{(1)}+\xi_j^{(N)}}{\xi_j^{(1)}-\xi_j^{(N)}}\right) \right|}^{-1} = \frac{|\xi_j^{(1)}-\xi_j^{(N)}|}{|\xi_j^{(1)}+\xi_j^{_(N)}|}.
			$$
		\end{prop}
		\begin{proof}
			The optimal polynomial for $p=1$ is given by \cite[Example 5.1, p.\,357]{opfer1984richardson}:
			\begin{equation*}
				\label{opt_pol_OS84_p1}
				\poly_1^{\ast}(\mu) = 1 - \frac{\mu}{|\xi_j^{(1)}| + |\xi_j^{(N)}| } \left( \frac{| \xi_j^{(1)} |}{\xi_j^{(1)}} +  \frac{|\xi_j^{(N)}|}{\xi_j^{(N)}}  \right), \quad \mu \in [\xi_j^{(N)},\xi_j^{(1)}],
			\end{equation*}
			and we have \cite[Eq. (5.6), p.\,358]{opfer1984richardson}
			$$
			\max_{\mu \in [\xi_j^{(N)},\xi_j^{(1)}]} | \poly_1^{\ast}(\mu) | = \frac{| \xi_j^{(1)}-\xi_j^{(N)}| }{ |\xi_j^{(1)}| + | \xi_j^{(N)}|  } < 1. 
			$$
			As $\Re(\xi_j^{(1)})\ne\Re(\xi_j^{(N)})$ it is straightforward to show that $\frac{| \xi_j^{(1)}-\xi_j^{(N)}| }{ |\xi_j^{(1)}| + | \xi_j^{(N)}|  } < \frac{|\xi_j^{(1)}-\xi_j^{(N)}|}{|\xi_j^{(1)}+\xi_j^{(N)}|}$ and hence for $p=1$ the Chebyshev polynomial is strictly non-optimal.
		\end{proof}

		The procedure to derive the optimal degree-$p$ polynomial, $\poly^\ast_p(z)$, is described in \cite[Corollary to Theorem 3.5]{freund1986class}.  Explicit expressions are available for $p=1$ and $p=2$, but for larger values of $p$ computing the extremal points and resulting polynomial requires the solution of a system of $p$ {nonlinear equations}.  This makes it more computationally expensive to use the optimal polynomial in place of the Chebyshev polynomials within an iterative method. Numerical experiments reveal very little difference in convergence for a  variety of choices of $A$ when using Chebyshev polynomials rather than the optimal polynomials. We therefore proceed to use CI as the basis for the numerical case study presented in Section \ref{sec:Numerics}, and use the notation $\mathcal{P}_{NC}$ to refer to general nested Chebyshev preconditioners of this form.

		\section{Applying the preconditioner via MINRES}\label{sec:SP}
		
		We now introduce an alternative approach for preconditioning \eqref{eq:blockproblem}, which involves re-formulating each complex sub-problem as a (generalized) saddle point system of twice the dimension, by writing 
		${x_j = \Re({x_j}) + i \;\Im({x_j})}$,
		${b_j = \Re({b_j}) + i \;\Im({b_j})}$. 
		This allows us to solve a purely real linear system, which can be desirable in some applications. We note that for $\lambda_{j'}$, $j' \in \{1,\ell/2+1\}$ the shifted system \eqref{eq:blockproblem} is already real.
		
		The corresponding saddle point problem for \eqref{eq:blockproblem} is then given by 
		\begin{equation}\label{eq:saddlesystem}
			\underbrace{\begin{pmatrix}
					\Im(\lambda_j)I_\sizeA & {A}-\Re(\lambda_j)I_\sizeA \\
					{A}-\Re(\lambda_j)I_\sizeA & -\Im(\lambda_j)I_\sizeA
			\end{pmatrix}}_{\mathcal{S}}\begin{pmatrix}
				\Re({x_j}) \\ -\Im({x_j})
			\end{pmatrix} =  \jt{\begin{pmatrix}
					-\Im({b_j})\\ \Re({b_j}) 
			\end{pmatrix}},
		\end{equation}
		i.e. featuring a matrix of the form 
		\begin{equation*}
			\begin{pmatrix}
				\phimat & \psimat\\ \psimat & -\phimat
			\end{pmatrix},
		\end{equation*}
		with $\phimat$ and $\psimat$ symmetric positive definite.

		As the matrix $\mathcal{S}$ is indefinite, we can no longer apply CI to solve \eqref{eq:saddlesystem}. We therefore propose applying MINRES with a block diagonal preconditioner. 
		We consider a preconditioner of the form \cite[Section 4.1]{Zulehner}
		\begin{equation}\label{eq:minresprecond}
			\mathcal{P}_{D} =\begin{pmatrix}
				\phimat + \psimat & 0 \\ 0 & \phimat+\psimat
			\end{pmatrix}.
		\end{equation}
		We supply a brief result on the eigenvalues of the preconditioned matrix, also discussed in \cite[Section 4.1]{Zulehner} for instance, as follows:
		\begin{thm}
			The eigenvalues of the preconditioned system $\mathcal{P}_{D}^{-1}\mathcal{S}$ lie in  $\big[-1,-\frac{1}{\sqrt{2}}\big] \cup \big[\frac{1}{\sqrt{2}},1\big]$.
		\end{thm}
		\begin{proof}
			We determine the eigenvalues by considering the eigenproblem:
			\begin{equation*}
				\begin{pmatrix}
					\phimat & \psimat\\
					\psimat & -\phimat
				\end{pmatrix}\begin{pmatrix}
					{v} \\ y
				\end{pmatrix} = \begin{pmatrix}
					\lamalt(\phimat+\psimat){v}\\ \lamalt(\phimat+\psimat)y
				\end{pmatrix}.
			\end{equation*}
			Multiplying out yields
			\begin{equation*}
				\phimat{v} + \psimat y = \lamalt \phimat{v} + \lamalt \psimat{v}, \qquad \psimat{v} - \phimat y = \lamalt \phimat y + \lamalt \psimat y.
			\end{equation*}
			{We recall that $\mu_{\min} > {\Re({\lambda_j})}$, as the shifted systems lie in the right-half plane. Therefore $\psimat$ is invertible.} Substituting and rearranging, then pre-multiplying by ${v}^{\top}$, yields:
			\begin{equation*}
				\psimat{v} + \phimat \psimat^{-1}\phimat {v} = \lamalt^2(\phimat+\psimat)\psimat^{-1}(\phimat+\psimat){v} \quad \Rightarrow \quad \lamalt^2 = \frac{{v}^{\top}(\psimat+\phimat \psimat^{-1}\phimat){v}}{{v}^{\top}(\phimat+\psimat)\psimat^{-1}(\phimat+\psimat){v}}.
			\end{equation*}
			This has the form of $S_2^{-1}S$ in \cite[Theorem 4]{pearson2012new}, with $\psimat = \frac{1}{\beta}M$, $\phimat = \frac{1}{\sqrt{\beta}}K$ ($M$ and $K$ denoting finite element mass and stiffness matrices, which could be replaced with alternative discretizations of the identity and negative Laplacian operators), and $\beta = 1$. Using this result, we have $\lamalt^2 \in \left[\frac{1}{2},1 \right]$, whereupon 
			taking square roots gives the result of the theorem statement. 
		\end{proof}
		{We refer to the full preconditioner that is applied to \eqref{eq:precondsystem} by the notation $\mathcal{P}_{SP}$ (for saddle point reformulation).}
		We discuss the computational aspects of applying the inner preconditioner $\mathcal{P}_{SP}$ in the following section.

		\section{Numerical experiments}\label{sec:Numerics}

		In this section we study the numerical performance of the new  preconditioners, namely $\mathcal{P}_{NC}$ and $\mathcal{P}_{SP}$, introduced in  Sections \ref{sec:ApproxPrecond} and \ref{sec:SP}, {respectively. The numerical performance analysis focuses on the number of outer iterations required for convergence and the computational cost measured in matrix--vector products with $A$.} We compare our proposed preconditioners against the `best-case' $\alpha$-circulant preconditioner $\mathcal{P}_{\alpha}$, and solving the unpreconditioned system $\mathcal{A}x= b$.

		The experimental framework follows the test case in \cite{weaver_correlation_2016,weaver2018time}.
		For all experiments we take ${A}$ to be a diffusion operator based on the shifted (negative) Laplacian in 2D with Dirichlet boundary conditions
		\begin{equation}\label{IminusLap}
			{A} = I_\sizeA - \frac{\nu}{h^2} \, \Lap
		\end{equation}
		for $\nu = \frac{D^2}{2\ell-4}$, where $D$ is the Daley lengthscale, which we take to be $0.2$, 
		$\Lap$ is the discretized Laplacian using a five-point finite difference stencil, $\ell$ as usual denotes the number of diffusion steps and ${ n_x = \frac{1}{h}-1}$
		is the number of spatial steps in each direction on the unit square minus one (hence, $\sizeA = n_x^2$). We take $\ell$ to be even, and investigate a range of values for $\ell$ and $n_x$. For this operator on a regular grid we can compute the eigendecomposition of ${A}$ analytically (see Appendix \ref{sec:Append}).

		The right-hand side, $b$, has the structure given in \eqref{eq:A} with $b_1$ being drawn from the standard normal distribution.
		As a large number of experiments consider the case $\ell = 10$,
		{{which is the  value used in  the operational ocean data assimilation configuration \cite{Chrust_2025}},} we introduce particular notation for the $10$-th roots of unity: $[1,\lambda_2,\lambda_3,-\overline{\lambda}_3,-\overline{\lambda}_2,-1,-\lambda_2,-\lambda_3,\overline{\lambda}_3,\overline{\lambda}_2]$ where
		$\lambda_2 = 0.8090 + 0.5878i$ and $\lambda_3 = 0.3090 + 0.9511i$.

		All experiments are performed in {\scshape Matlab} version 2024b on a machine with a 2.5GHz Intel
		sixteen-core i7 processor with 32GB RAM on an Ubuntu 24.04.1 LTS operating system. The stopping criteria for CI is based on the two-norm  of the relative residual $r_p = \frac{\|b-\mathcal{A}\chi_p\|_2}{\|b\|_2}$, as we take the initial guess $\chi_0 = 0$.  Unless otherwise specified CI is terminated when $r_p < 10^{-6}$.
		
		\subsection{Implementational concerns}
		
		We now discuss methods to implement each of our new preconditioners in order to control the number of matrix--vector products with $A$ for each iteration of CI applied to \eqref{eq:precondsystem}, which is the dominant computational expense for this problem.    
		The total permitted number of matrix--vector products with $A$ used to apply \jt{the nested Chebyshev} preconditioners is given by $\ell n_x  \resource$, where $\resource$ is a user-defined parameter. We investigate a range of values, $\resource \in [0.1,1]$. In this paper $n_x$ is used to scale the resource to allow for fair comparison across experiments of different dimensions. We define the total computational  budget in terms of the maximum number of matrix--vector products with $A$ as given in Table~\ref{tab:compbud}. {We allocate the same maximum number of inner iterations to all methods, which results in different computational budgets for $\mathcal{P}_{NC}$ and $\mathcal{P}_{SP}$.}  For $\mathcal{P}_{NC}$, two different allocation methods are studied: 
		\begin{itemize}
			\item $\mathcal{P}_{NC1}$ allocates the same budget to each of the $\ell$ sub-problems, leading to varying accuracy of the solution across the sub-problems,
			\item $\mathcal{P}_{NC2}$ allocates the budget according to Algorithm~1. This approach is based on the upper bound on the asymptotic convergence factor (as given in Theorems \ref{thm:AxelssonIteration} and \ref{thm:convergence}) for each sub-problem, ensuring similar accuracy of the computed solution across all sub-problems.
		\end{itemize}

		Evaluating Algorithm 1 requires only scalar operations and is done ahead of applying CI to \eqref{eq:precondsystem}. This leads to negligible additional setup cost for $\mathcal{P}_{NC2}$ compared to $\mathcal{P}_{NC1}$ \jt{so long as approximations to the extreme eigenvalues of $A$, i.e. $\mu_1$ and $\mu_N$, are already available}.  For a large value of $\resource$,  we expect $\mathcal{P}_{NC1}$ and $\mathcal{P}_{NC2}$ to have the same performance. {We could further accelerate convergence of $\mathcal{P}_{NC1}$ and $\mathcal{P}_{NC2}$ by additionally preconditioning the sub-problems \eqref{eq:blockproblem}. However, this would require good estimates of the extreme eigenvalues of the preconditioned system, with the quality of these estimates affecting the convergence of CI applied to \eqref{eq:precondsystem}. A number of simple preconditioners (e.g. preconditioning with the diagonal) were found to be ineffective for this problem, according to the results of \cite{weaver2018time}. We hence apply unpreconditioned CI to \eqref{eq:blockproblem} to perform matrix--vector products with $\mathcal{P}_{NC1}^{-1}$ and $\mathcal{P}_{NC2}^{-1}$ using equation~\eqref{eq:Paform}.
			
			\begin{table}[ht]
				\centering
				\begin{tabular}{c|ll}
					Name {/ Operator} & $\mathcal{A}(\cdot)$ & $\mathcal{P}^{-1}(\cdot)$  \\ \hline
					$I_{\ell N}$ & $\ell$ & --  \\
					$\mathcal{P}_{NC}$ & $\ell$ & $\ell n_x  \resource$ \\
					$\mathcal{P}_{SP}$ & $\ell$ & ${2(\ell-1) }n_x  \resource$ and $\ell$ AMG 
				\end{tabular}
				\caption{{Maximum number of matrix--vector products with $A$  within} one outer iteration of CI applied to \eqref{eq:precondsystem} for different choices of preconditioners. For $\mathcal{P}_{SP}$ this {also includes} the  number of AMG initializations required to apply the preconditioner.}
				\label{tab:compbud}
			\end{table}

			\begin{algorithm}[ht]
				\caption{Algorithm for assigning iteration resource unevenly across sub-problems, according to the upper bound on the asymptotic convergence factor given in Theorem \ref{thm:convergence}}\label{alg:AssignIterations}

				\textbf{Input:} $\ell$, $n_x$, $\umax$, $\umin$, $\alpha$, $\resource$ \\
				\textbf{Output:} {$\texttt{all}$ resource allocation for sub-problems}\\
				1: Compute $\ell$ scaled roots of unity, $\lambda_1,\dots,\lambda_{\ell}$.
				
				2: \textbf{for }{$j = 1,\dots,\ell$}{\\
					3: \,\quad Compute $\sigma_{{j}} = \frac{\sqrt{\kappa(\Re(B_j))}-1}{\sqrt{\kappa(\Re(B_j))}+1}$.\\
					4: \,\quad Compute $r({{j}}) = \frac{\ln(\sigma_1)}{\ln(\sigma_{{j}})}$.
				}\\
				5: \textbf{end for}\\
				6: Normalize $r = \frac{r}{\sum_{{j}}r{(j)}}$. \\
				7: Allocation is $\texttt{all} = \lfloor r  \ell n_x \resource\rfloor$ where the floor operation is applied entrywise.
			\end{algorithm}

			For $\mathcal{P}_{SP}$  we replace \eqref{eq:blockproblem} with an equivalent problem entirely in real arithmetic, \eqref{eq:saddlesystem}. {We allocate the same number of iterations to each sub-problem}. For real roots $\lambda_{j'}$, $j'\in \{1,\ell/2+1\}$ we solve 
			$({{A}}\pm{\alpha^{1/\ell}}{I}_\sizeA) x_{j'} = b_{j'}$  using the preconditioned Conjugate Gradient method (see \cite{HestenesStiefel}) with an algebraic multigrid (AMG) \cite{ruge1987algebraic} preconditioner that approximates $A \pm \alpha^{1/\ell} I_N$. For complex roots $\lambda_{j}$, $j \in [2,\ell/2]\cup [\ell/2+2,\ell]$, we solve \eqref{eq:saddlesystem} using MINRES, {which requires two applications of $A$ per iteration}. We use AMG to apply the block diagonal preconditioner $\mathcal{P}_{D}$, given by \eqref{eq:minresprecond}, approximately.
			We apply $\mathcal{P}_{D}^{-1}$ blockwise, using the \texttt{HSL\_MI20} algebraic multigrid solver \cite{boyle2010hsl_mi20} as a black box with default parameters, using a single V-cycle. In our numerical experiments coarsening often terminates prematurely, suggesting that even better numerical performance could be obtained by adapting the AMG implementation to the problem of interest.  
			A practical application of the preconditioner $\mathcal{P}_{SP}$ depends on having a high-quality and computationally efficient implementation of AMG (or similarly affordable way) to apply $\mathcal{P}_{D}^{-1}$. This is not trivial for many problems, and may limit the applicability of this approach for general systems.

			We note that since our new preconditioners only approximate $\mathcal{P}_\alpha$, the spectral bounds of Corollary~\ref{cor:extremeevals} are no longer guaranteed to hold.  In what follows we do not adjust the spectral limits given by Corollary \ref{cor:extremeevals} when applying $\mathcal{P}_{NC1}$, $\mathcal{P}_{NC2}$, and $\mathcal{P}_{SP}$. This is likely to have the most significant effect for small choices of $\resource$.

			\begin{figure}
				\centering
				\includegraphics[width=0.95\textwidth]{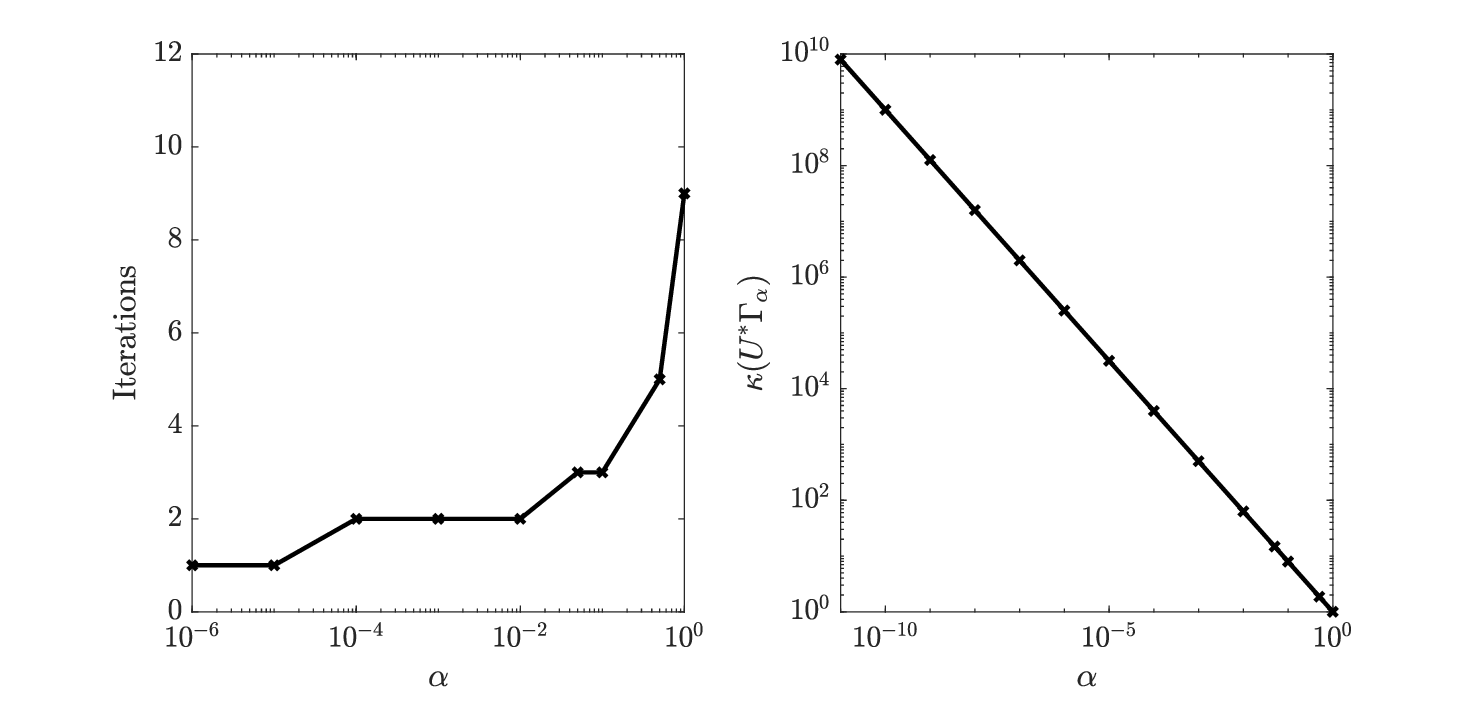}
				\caption{Left panel: Number of iterations to reach convergence for the problem introduced in Section \ref{sec:Numerics} using the block $\alpha$-circulant preconditioner $\mathcal{P}_{\alpha}$ for $\ell=10$, $n_x=100$, and hence $\mathcal{A}\in \mathbb{R}^{100,000 \times 100,000}$. Right panel: Condition number of the matrix of eigenvectors of $\mathcal{P}_{\alpha}$, $U^*\Gamma_{\alpha}$, for different values of $\alpha$.}
				\label{fig:exact alpha circ}
			\end{figure}
			
			The `best-case' preconditioner $\mathcal{P}_{\alpha}$ is applied using the backslash command in {\scshape{MATLAB}}. It is not  straightforward to express the cost of a single application of $\mathcal{P}_{\alpha}^{-1}$ in terms of matrix--vector products with $A$, so we simply compare the performance in terms of outer iterations. 
			The left panel of Figure \ref{fig:exact alpha circ} shows the number of iterations required to reach convergence for $\mathcal{P}_{\alpha}$ for a range of values of $\alpha$. For $\alpha=1$, we recover the block-circulant preconditioner. The number of iterations decays monotonically with $\alpha$, with convergence in one iteration occurring for $\alpha \le 10^{-5}$.  
			When using the `best-case' preconditioner the fastest convergence is obtained when taking $\alpha$ as small as possible. \jt{We note that the eigenvalues of $\mathcal{P}_{\alpha}^{-1}\mathcal{A}$ become more clustered} as $\alpha$ decreases (see Corollary \ref{cor:extremeevals}). 
			However, it is known that the eigenvector matrix of $\mathcal{P}_{\alpha}$, $U^*\Gamma_{\alpha} $, can become very ill-conditioned in the case of small $\alpha$ \cite{gander_palitta_2024}, which is seen in the right panel of Figure \ref{fig:exact alpha circ} for our test problem. There is hence a trade-off between ensuring fast convergence of the preconditioned problem and avoiding ill-conditioning within the preconditioner itself.
			\subsection{Performance of nested CI approaches, $\mathcal{P}_{NC1}$ and $\mathcal{P}_{NC2}$}\label{sec:NumericsCI}

			We now consider the performance of preconditioners based on nested CI, namely $\mathcal{P}_{NC1}$ and $\mathcal{P}_{NC2}$. We begin by considering the performance of CI applied to the sub-problems~\eqref{eq:blockproblem}. 
			Table \ref{tab:Cheb conv tol} shows the number of iterations of CI required for the sub-problems $A-\lambda_j I_\sizeA$  to reach convergence in the case $\ell=10$, $n_x=100$, and $\alpha=1$. We see that, in agreement with the theory of Section \ref{sec:ApproxPrecond}, the iterations for conjugate pairs $\lambda_j$, $\overline{\lambda}_j$ are the same. Additionally, for a fixed tolerance, significantly more iterations are required for the case $\lambda_j=1$ than for the case $\lambda_j = -1.$ This indicates that if the same number of iterations are allocated to each sub-problem, then the blocks are solved to different tolerances. For small $\resource$, the poor convergence associated with $\lambda_j=1$ may lead to larger numbers of outer iterations (indeed this is observed in subsequent numerical experiments). 
			Following the result of Theorem \ref{thm:convergence}, we also expect the convergence behaviour for the sub-problems to become more similar with decreasing $\alpha$, due to smaller differences in the values of $\Re(\lambda_j)$.}
		
		\begin{table}[htb]
			\centering
			\begin{tabular}{c|cccccccccc}
				$\epsilon$  &   $1$ & $\lambda_2$ & $\lambda_3$ & $-\overline{\lambda}_3$& $-\overline{\lambda}_2$ &$  -1$& $-\lambda_2$ & $-\lambda_3$ & $\overline{\lambda}_3$ & $\overline{\lambda}_2$ \\
				\hline
				$10^{-10}$ &760 &  274 &  184 &  147  & 128 &  118 &  128 &  147 &  184 &  274\\
				$10^{-9}$&683 &  248 &  167 &  133 &  115 &  107 &  115 &  133 &  167 &  248\\
				$10^{-8}$&611 &  222 &  150 &  119 &  103 &   95 &  103 &  119 &  150 &  222\\
				$10^{-7}$&535 &  196 &  132 &  105 &   90 &   84 &   90 &  105 &  132 &  196\\
				$10^{-6}$&463 &  170 &  114 &   90 &   78 &   72  &  78 &   90 &  114 &  170\\
				$10^{-5}$&388 &  143 &   97 &   76 &   65 &   61  &  65 &   76 &   97 &  143\\
				$10^{-4}$&314 &  116 &   79 &   62 &   53 &   49  &  53 &   62 &   79 &  116\\
				$10^{-3}$&240 &   89 &   60 &   47 &   40 &   38  &  40 &   47 &   60 &   89\\
				$10^{-2}$&166 &   62 &   42 &   33 &   28 &   26  &  28 &   33  &  42 &   62\\
				$10^{-1}$& 93&    34 &   23 &   18 &   15 &   15  &  15 &   18  &  23 &   34
			\end{tabular}
			\caption{Number of iterations to reach convergence for the inner sub-problems \eqref{eq:blockproblem}, involving $A-\lambda_j I_\sizeA$, using CI for different choices of tolerance, $\varepsilon$. For this problem $\ell=10$, $n_x=100$, and $\alpha=1$.}
			\label{tab:Cheb conv tol}
		\end{table}
		
		Figure~\ref{fig:Asymptotic rate of convergence} shows the relative error $r_p = \frac{\|b-\mathcal{A}\chi_p\|_2}{\|b\|_2}$ at each iteration of CI applied to the sub-problems \eqref{eq:blockproblem} for $\alpha=1$ (left) and $\alpha=0.01$ (right). The dashed lines show an upper bound on the relative error at each iteration, which is obtained by multiplying the initial error by the upper bound on the asymptotic  convergence factor given by Theorem \ref{thm:convergence}. 
		We note that this upper bound is tightest for $\lambda_j$ with smallest real part, and for smaller values of $\alpha$. Using  Algorithm 1 to allocate computational resource across the sub-problems is therefore likely to overallocate resource to $\lambda_j$ with positive real part in the case $\alpha=1$ to a greater extent than for  $\alpha=0.01$.
		
		\begin{figure}
			\centering
			\includegraphics[width= 0.95\textwidth,trim =0mm 0cm 10mm 0cm, clip]{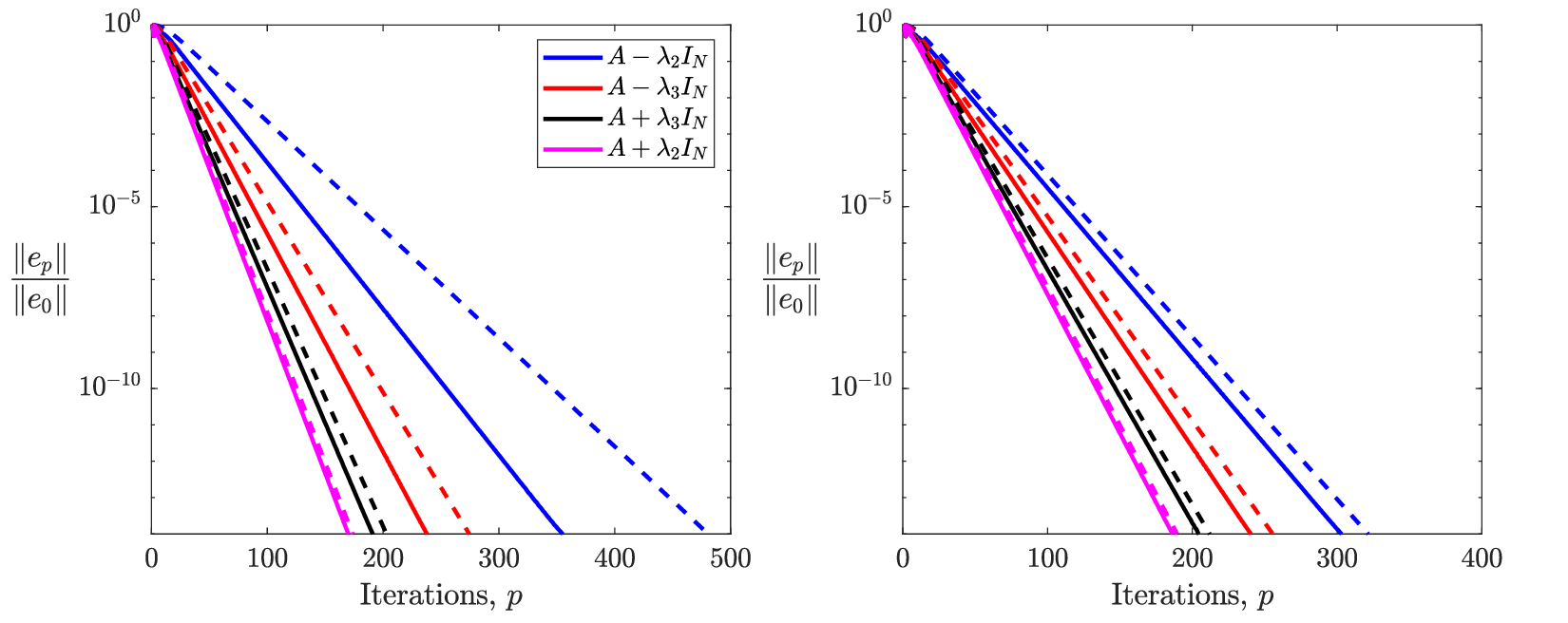}
			\caption{Relative error at each iteration of CI applied to { $A\pm\lambda_j I_N$} for $\lambda_j$ complex, $\ell=10$, $n_x = 100$ (solid lines) and upper bound on the error using the asymptotic  convergence factor given by Theorem \ref{thm:convergence} (dashed lines), for $\alpha=1$ (left) and $\alpha = 0.01$ (right). Only complex roots with positive imaginary part are shown.}
			\label{fig:Asymptotic rate of convergence}
		\end{figure}

		\begin{figure}
			\centering
			\includegraphics[width=0.95\textwidth]{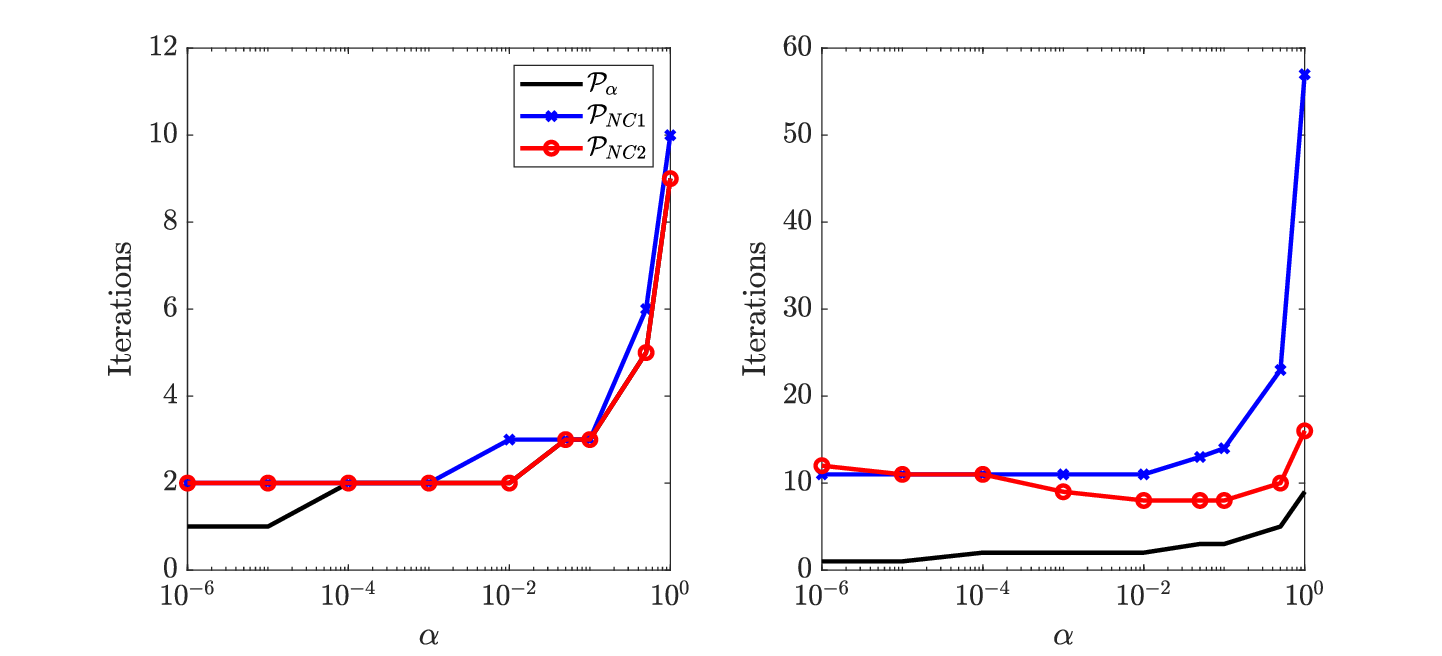}
			\caption{Number of iterations to reach convergence using $\mathcal{P}_{\alpha}$ (black), $\mathcal{P}_{NC1}$ (blue crosses), and $\mathcal{P}_{NC2}$ (red circles) for different values of $\alpha$ and for $\resource = 1$ (left) and $\resource= 0.2$ (right), with $\ell=10$ and $n_x=100$. Note the difference in $y$-axis values between the two plots.}
			\label{fig:alphaNC1NC2}
		\end{figure}
		We now investigate the performance of $\mathcal{P}_{NC1} $ and $\mathcal{P}_{NC2}$ applied to \eqref{eq:precondsystem}. We start by considering the sensitivity of the preconditioners to the choice of $\alpha$ for different values of $\resource$. 
		The left panel of Figure \ref{fig:alphaNC1NC2} shows the number of iterations required to reach convergence in a high resource setting ($\resource = 1$). Performance for $\mathcal{P}_{NC1}$ (blue crosses) and $\mathcal{P}_{NC2}$ (red circles) is similar both to each other, and the best-case preconditioner $\mathcal{P}_{\alpha}$ (black line). 
		In this high resource setting,  the majority of the sub-problems \eqref{eq:blockproblem} are solved to a tolerance at least as strict as $10^{-5}$ (see Table \ref{tab:Cheb conv tol} for an illustration for the case $\alpha = 1$). This  is close to the tolerance of $10^{-6}$ for the outer CI applied to \eqref{eq:precondsystem}. 
		Therefore there is little gain to be made by adopting the load balancing approach of $\mathcal{P}_{NC2}$. 
		By increasing $\eta$ further, or using a lower tolerance to solve the sub-problems \eqref{eq:blockproblem}, it is possible to match the performance of $\mathcal{P}_\alpha$ with $\mathcal{P}_{NC1}$ and $\mathcal{P}_{NC2}$ (not shown).

		The right panel of Figure~\ref{fig:alphaNC1NC2} shows the number of iterations required to reach convergence for $\mathcal{P}_{NC1}$ and $\mathcal{P}_{NC2}$ for varying $\alpha$ in a small resource setting ($\resource=0.2$). In this setting, and for $\alpha>10^{-4}$, there is a clear advantage to using $\mathcal{P}_{NC2}$. The largest number of {outer} iterations is required for $\alpha = 1$, with $\mathcal{P}_{NC1}$ requiring 57 iterations compared to 16 for $\mathcal{P}_{NC2}$. 
		For $\mathcal{P}_{NC1}$ each sub-problem is allocated $20$ inner iterations, resulting in \eqref{eq:blockproblem} being solved to very different tolerances: e.g. \aw{approximately} $0.04$ for $\lambda_j=-1$ and $0.72$ for $\lambda_j=1$ when $\alpha=1$. 
		In contrast, for $\mathcal{P}_{NC2}$ Algorithm 1 assigns more iterations to the sub-problem corresponding to $\lambda_j=1$ and fewer to the case for $\lambda_j=-1$ ($30$ and $4$ respectively when $\alpha=1$). This means that all sub-problems are solved to approximately the same tolerance ($0.6$ for $\alpha=1)$, resulting in better overall performance of the preconditioner.  For $\alpha \ll 1$ there is less advantage to using $\mathcal{P}_{NC2}$, in agreement with the conclusions of Theorem \ref{thm:convergence}. For $\mathcal{P}_{NC2}$, the smallest number of iterations is required for $\alpha = 0.01$.  In the experiments that follow we consider the two cases $\alpha = 1$ and $0.01$.

		\begin{table}[]
			\centering
			\begin{tabular}{c|ccc|ccc}
				
				\multirow{2}{*}{\Large $n_x \backslash {\resource}$ } & \multicolumn{3}{c|}{$\mathcal{P}_{NC1}$} & \multicolumn{3}{c}{$\mathcal{P}_{NC2}$}\\   
				& 0.1&0.2&0.3 & 0.1&0.2&0.3 \\
				\hline
				$\alpha = 1$&&&&&&\\ \hline
				50  & 164	(9840) & 62	(6820) &	35	(\textbf{5600}) &	70	(3710) &	20	(2040)	&12	(\textbf{1848})\\
				100 & 140	(15400) &	56	(11760) &	33	(\textbf{10230})&	47	(4794)	& 16	(\textbf{3248})&	11	(3355)\\
				200& 113	(23730)	&51	(20910) &	28	(\textbf{17080})&	37	(7511)&	15	({6075})	&10	(\textbf{6040})\\
				300& 103	(31930)&	46	(28060)&	27	(\textbf{24570})&	34	(10370)&	14	(\textbf{8456})&	10	(9050)\\
				400 & 92	(37720)&	44	(35640)&	26	(\textbf{31460}) &	31	(12555) &	13	(\textbf{10452})&	10	(12050)\\
				
				500 &  87	(44370)&	40	(40400)&	25	(\textbf{37750})&	30	(15150) &	12	(\textbf{12060})&	10	(15040)\\
				\hline
				$\alpha = 0.01$&&&&&&\\ \hline          
				
				50 &38	(2280)&	13	(1430) &	7	(\textbf{1120}) &	27	(1512)	&10	(\textbf{1050})	&7	(1085)\\
				100 & 33	(3630)&	12	(2520)&	7	(\textbf{2170})&	21	(2205)&	8	(\textbf{1640}) &	6	(1824) \\
				200&31	(6510)&	11	(4510)&	6	(\textbf{3660}) &	19	(3895) &	8	(\textbf{3240})&	6	(3636)\\
				300& 30	(9300) &	10	(6100) &	6	(\textbf{5460}) &	18	(5472) &	7	(\textbf{4242}) &	5	(4520)\\
				400 & 29	(11890)	&10	(8100)&	6	(\textbf{7260}) &	18	(7290) &	7	(\textbf{5628})	& 5	(6030)\\
				500 & 28	(14280)&	10	(10100) &	6	(\textbf{9060}) &	17	(8602)	&7	({7035}) &	4	(\textbf{6020})
			\end{tabular}

			\caption{Number of iterations to reach convergence of CI applied to \eqref{eq:precondsystem} using $\mathcal{P}_{NC1}$ and $\mathcal{P}_{NC2}$ with increasing size of the sub-problem, $n_x$, with $\ell = 10$ sub-problems. The total number of matrix--vector products with $A$ required to reach convergence is given in brackets. The permitted computational resource used to apply the preconditioner at each iteration is given by $\ell n_x  \resource$. The problem size ranges from $\mathcal{A}\in \mathbb{R}^{25,000\times 25,000}$ to $\mathcal{A}\in \mathbb{R}^{2,500,000\times 2,500,000}$.} 
			\label{tab:nx_scaling}
		\end{table}

		For the motivating diffusion-based covariance problem in \cite{weaver_correlation_2016, weaver2018time}, $\sizeA$ is very large (order of $10^5$ and greater in those studies), so it is of interest to understand how our methods scale with $n_x$, the number of points in each spatial dimension. We note that for this problem the extreme eigenvalues of $\mathcal{P}_\alpha^{-1}\mathcal{A}$ do not change as $n_x$ increases (see Corollary \ref{cor:extremeevals}). Table \ref{tab:nx_scaling} shows the change in the number of  iterations of CI applied to \eqref{eq:precondsystem} and matrix--vector products with $A$ needed when using $\mathcal{P}_{NC1}$ and $\mathcal{P}_{NC2}$ for increasing problem size for $\alpha = 1$ and $\alpha = 0.01$. Here we increase $\eta$ in line with the dimension of the system, considering $\resource =  [0.1,0.2,0.3]$. This allows us to obtain scale independence of the outer iterations for both preconditioning approaches in the case $\alpha = 0.01$. As the number of  iterations is reduced by using a smaller value of $\alpha$, we  obtain a corresponding large reduction in the number of matrix--vector products for $\alpha = 0.01$ compared to $\alpha=1$. Increasing the resource improves the performance of the preconditioner in terms of outer iterations, but for $\mathcal{P}_{NC2}$ the smallest number of matrix--vector products is {often} achieved for $\resource=0.2$. This is because the reduction in outer iterations for $\eta = 0.3$ is not sufficient to mitigate for the extra cost of applying the preconditioner compared to $\eta=0.2$. For both values of $\alpha$ considered here  $\mathcal{P}_{NC2}$ performs better than $\mathcal{P}_{NC1}$ in terms of  iterations and total matrix--vector products with ${A}$, although the difference is larger for $\alpha = 1$.

		\begin{table}[]
			\centering
			\begin{tabular}{c|ccc|ccc}
				\multirow{2}{*}{\Large $\ell \backslash {\resource}$ } &\multicolumn{3}{c|}{$\mathcal{P}_{NC1}$} & \multicolumn{3}{c}{$\mathcal{P}_{NC2}$}\\   
				& 0.1&0.2&0.3 & 0.1&0.2&0.3 \\
				\hline
				$\alpha=1$&&&&&&\\ \hline
				
				6  &    118  (7788) &  50 (6300)&  28 (\textbf{5208})  & 58 (3596) & 18 (2214) & 11 (\textbf{2002})\\
				10  & 135	(14850)	&57	(11970)&	33	(\textbf{10230})	&48	(4896)&	17	(3451)&	11	(\textbf{3355}) \\
				16 & 152	({26752})	&61	(20496)	&35	(\textbf{17360})	&35	(5915)	&14	(\textbf{4592})	&11	(5379) \\
				20 & 163	(35860)	&64	(26880)&	35	(\textbf{21700})	&31	(6603)	&14	(\textbf{5768})	&11	(6721)\\
				30 & 168	(55440)&	65	(40950)	&37	(\textbf{34410})	&26	({8320})	&13	({\textbf{8034}})	&11	(10120)\\
				40 &  176	(77440)	&66	(55440)	&37	(\textbf{45880})	&27	({11313})	&12	(\textbf{9852})	&11	(13387)\\
				50& 185	(101750)	&71	(74550)	&39	(\textbf{60450})	&24	({12576}) &	12	(\textbf{12300})	&11	(16786)\\
				\hline
				$\alpha=0.01$&&&&&&\\ \hline   
				6 &  41	(2706)	&13	(1638)&	8	(\textbf{1488})	&32	(1984)&	10	(1230)	&6	(\textbf{1104})   \\
				10&34	(3740)	&12	(2520)&	7	(\textbf{2170})	&21	(2205)	&8	(\textbf{1640})	&6	(1824) \\
				16&31	(5456)	&11	(3696)&	7	(\textbf{3472})	&17	(2856)&	8	(2640)	&5	(\textbf{2445}) \\
				20& 30	(6600)	&11	(4620)&	7	(\textbf{4340})	&15	(3135)&	8	(3272)	&4	(\textbf{2440})\\
				30&29	(9570)	&10	(6300)&	6	(\textbf{5580})	&18	(5706)&	6	({3690})	&4	(\textbf{3664}) \\
				40& 28	(12320)	&10	(8400)&	6	(\textbf{7440})	&20	(8420)&	6	({4938})	&4	(\textbf{4884}) \\
				50&  28	(15400)	&10	(\textbf{10500})&	7	(10850)	&20	(10500)	&6	({6144})	&4	(\textbf{6112}) \\
			\end{tabular}
			\caption{Number of iterations to reach convergence of CI applied to \eqref{eq:precondsystem} using $\mathcal{P}_{NC1}$ and $\mathcal{P}_{NC2}$ with increasing number of  sub-problems, $\ell$, with $n_x = 100$. The total number of matrix--vector products with $A$ required to reach convergence is given in brackets. The permitted computational resource used to apply the preconditioner at each iteration is given by $\ell n_x \resource$. The problem size ranges from $\mathcal{A}\in \mathbb{R}^{60,000\times 60,000}$ to $\mathcal{A}\in \mathbb{R}^{500,000\times 500,000}$. \label{tab:NCscalingwithl}}
		\end{table}

		Table~\ref{tab:NCscalingwithl} shows how the number of  iterations and matrix--vector products with ${A}$ change as we increase the number of blocks in $\mathcal{A}$ for $\alpha = 1$ and $\alpha=0.01$. We note that for the motivating covariance problem only small values of $\ell$ are used, but scaling with $\ell$ may be relevant for other applications.  Similarly to Table \ref{tab:nx_scaling}, $\mathcal{P}_{NC2}$ requires substantially fewer iterations to reach  convergence than $\mathcal{P}_{NC1}$.  For both methods, choosing $\alpha = 0.01$ yields fewer outer iterations and matrix--vector products with $A$ than $\alpha =1$. For an appropriate choice of $\eta$, we can achieve a (near) scale-independent number of outer iterations with increasing $\ell$. This property can be linked with the theoretical results concerning the spectrum of $\mathcal{P}_{\alpha}^{-1}\mathcal{A}$ from Section \ref{sec:Background}. Theorem \ref{thm:evals} states that  $\mathcal{P}_\alpha^{-1}\mathcal{A}$ has $(\ell-1)N$ unit eigenvalues and $N$ potentially non-unit eigenvalues. Hence, as $\ell$ increases, the proportion of non-unit eigenvalues decreases. Additionally, as $\umin>1\ge \alpha$ for this experiment, the expression in Corollary \ref{cor:extremeevals} is monotonically decreasing with increasing $\ell$. Therefore the largest eigenvalue of $\mathcal{P}_\alpha^{-1}\mathcal{A}$ decreases as $\ell$ increases, with $\lim_{\ell \rightarrow \infty} \lambda_{\max}(\mathcal{P}_\alpha^{-1}\mathcal{A}) = 1$. 
		We note that for  smaller values of $\eta$ we are further away from the theoretical setting. 
		Although the number of  iterations is small, the number of matrix--vector products with ${A}$ increases approximately linearly with $\ell$, as the number of sub-problems to be solved increases.

		\begin{figure}
			\centering
			\includegraphics[width=\textwidth,trim = 20mm 0mm 20mm 0mm, clip]{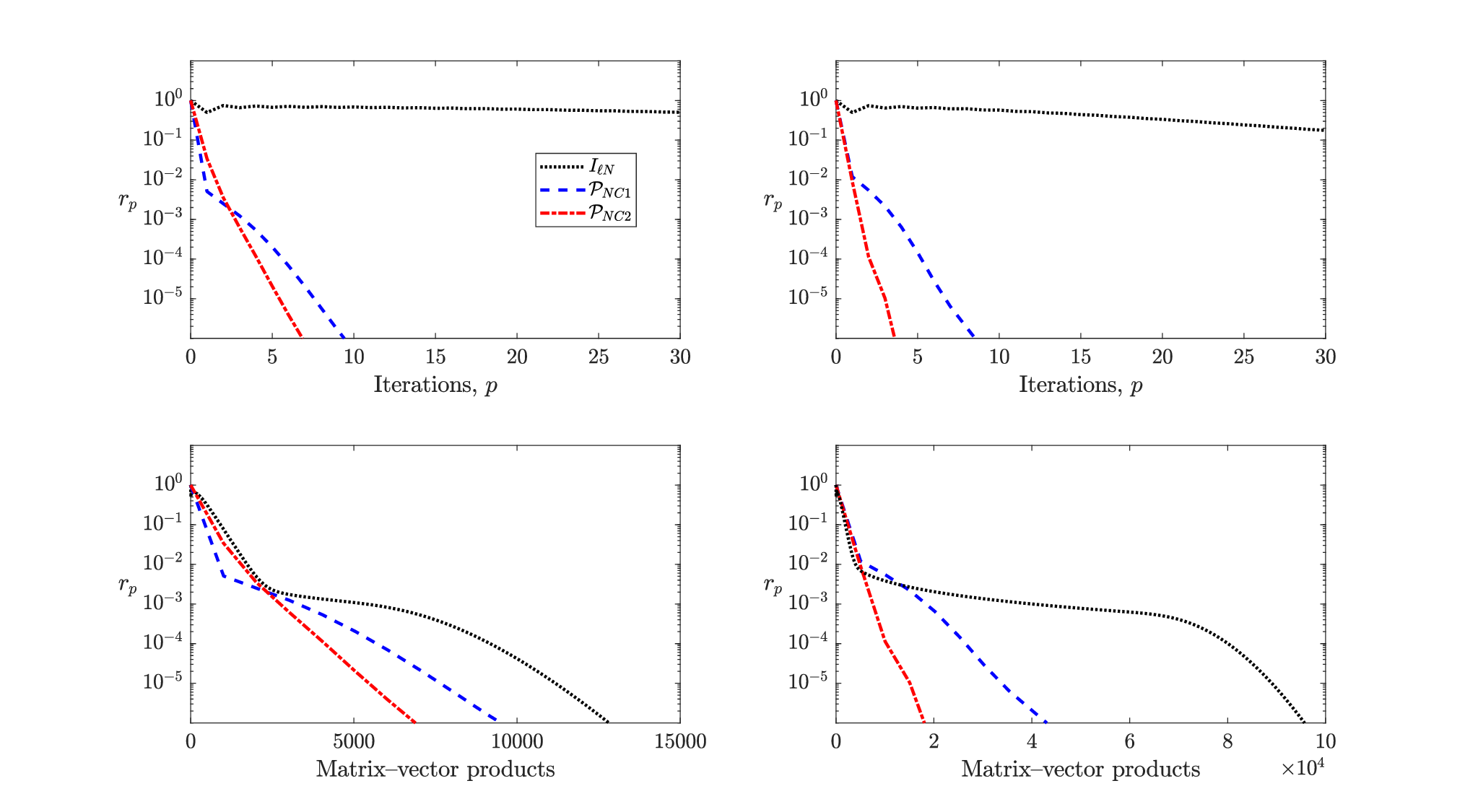}
			\caption{Comparison of the  norm of the relative residual $r_p=\frac{\|b-\mathcal{A}\chi_p\|_2}{\|b\|_2}$ with iterations (top panels) and matrix--vector products (bottom panels) using  $I_{\ell N}$, $\mathcal{P}_{NC1}$, and $\mathcal{P}_{NC2}$ for $\resource = 0.2$ and $\alpha=0.01$. The left panels show the case $\ell = 10$ and  the right panels show $\ell=50$.  For this problem $n_x=500$. Note the difference in the $x$-axis scales between the lower left and right panels.}
			\label{fig:residual_vs_matvecs}
		\end{figure}
		
		We now compare the performance of the $\mathcal{P}_{NC}$ preconditioners against solving the unpreconditioned system.
		Figure \ref{fig:residual_vs_matvecs} shows the reduction in the relative norm of the residual $r_p := \frac{\|b-\mathcal{A}\chi_p\|_2}{\|b\|_2}$ plotted against iterations (top panels) and matrix--vector products with $A$ (bottom panels) for a number of preconditioners for $\ell=10$ (left) and $\ell=50$ (right) when $\alpha=0.01$ and $n_x=500$. The matrix--vector products include both the application of $\mathcal{A}$ and the relevant preconditioner.   Solving the unpreconditioned system requires a much larger number of iterations to reach convergence ($1282$ iterations for $\ell=10$, and $1914$ for $\ell=50$). This means that solving \eqref{eq:precondsystem} with either $\mathcal{P}_{NC1}$ or $\mathcal{P}_{NC2}$ requires fewer matrix--vector products with $A$ than solving the unpreconditioned system \eqref{eq:A}. Additionally, the load balancing approach of $\mathcal{P}_{NC2}$ results in a clear reduction in outer iterations and matrix--vector products with $A$ compared to $\mathcal{P}_{NC1}$ for both choices of $\ell$. As $n_x$ and $\ell$ increase we expect further computational gains due to the beneficial scaling behaviour of $\mathcal{P}_{NC2}$. \jt{Tests with matrices $A$ that cannot be readily diagonalized showed qualitatively similar behaviour to that described in this section, indicating that the conclusions presented here are likely to apply to a broader range of applications. }
		
		\subsection{Performance of saddle point reformulation, $\mathcal{P}_{SP}$}
		We now study the performance of $\mathcal{P}_{SP}$, the preconditioner introduced in Section \ref{sec:SP}. This method transforms the sub-problems corresponding to complex roots of unity  into a saddle point system which takes real values. We begin by considering the performance of MINRES with AMG applied to the sub-problems \eqref{eq:saddlesystem}. Table \ref{tab:SP conv tol} shows the number of iterations required for each inner problem to reach convergence for each of the sub-problems \eqref{eq:blockproblem} when applying $\mathcal{P}_{D}$ with a large resource level ($\resource=1$). We recall that the two real scaled roots of unity are solved using the preconditioned Conjugate Gradient method. For this implementation these real roots require the fewest iterations to reach convergence.  There is a slight difference between the convergence of the sub-problems corresponding to complex $\lambda_j$, but this is not as dramatic as in the previous section for $\mathcal{P}_{NC}$. We also note that the numbers of iterations required to reach a given tolerance are much smaller when using $\mathcal{P}_{D}$ than $\mathcal{P}_{NC}$. We therefore do not propose a resource allocation version of this method.

		\begin{table}[htb]
			\centering
			\begin{tabular}{c|cccccccccc}
				$\epsilon$  &   $1$ & $\lambda_2$ & $\lambda_3$ & $-\overline{\lambda}_3$& $-\overline{\lambda}_2$ &$  -1$& $-\lambda_2$ & $-\lambda_3$ & $\overline{\lambda}_3$ & $\overline{\lambda}_2$ \\
				\hline
				$10^{-10}$ & 7   & 28 &   28 &   28  &  24 &    6 &   24 &   28 &   28 &   28\\
				$10^{-9}$&6   & 26  &  26   & 24  &  22  &   5  &  22  &  24  &  26  &  26\\
				$10^{-8}$&6   & 24  &  22  &  22 &   18  &   5   & 18  &  22  &  22  &  24\\
				$10^{-7}$& 5  &  20  &  20  &  20  &  16  &   4  &  16 &   20 &   20 &   20\\
				$10^{-6}$& 5   & 18   & 18  &  18  &  14  &   4  &  14 &   18   & 18  &  18\\
				$10^{-5}$&   4  &  16 &   16 &   14 &   12 &    3 &   12&    14  &  16 &   16\\
				$10^{-4}$& 3   & 12   & 12  &  12  &  10  &   3   & 10  &  12   & 12  &  12\\
				$10^{-3}$&  3  &  10  &  10 &   10  &   8  &   2  &   8 &   10  &  10 &   10\\
				$10^{-2}$& 2   &  8   &  6  &   6  &   6   &  2   &  6   &  6   &  6   &  8\\
				$10^{-1}$&   2  &   5 &    4  &   4  &   4  &   1 &    4  &   4  &   4  &   5
			\end{tabular}
			\caption{Number of iterations to reach convergence for the inner sub-problems \eqref{eq:blockproblem}, involving $A-\lambda_j I_\sizeA$, using the preconditioner $\mathcal{P}_{D}$ for different choices of tolerance. For this problem $\ell=10$, $n_x=100$, and $\alpha = 1$.}
			\label{tab:SP conv tol}
		\end{table}

		We now consider the performance of $\mathcal{P}_{SP}$ applied to \eqref{eq:precondsystem}. 
		{The results are the same} for the high and low resource limits considered in Figure \ref{fig:alphaNC1NC2}, unlike for $\mathcal{P}_{NC}$, and we hence only describe the results for the low resource limit. The iterations match those of the `best-case' preconditioner $\mathcal{P}_{\alpha}$ for $\alpha \ge 10^{-4}$ (see Figure \ref{fig:alphaNC1NC2}). However,  \eqref{eq:precondsystem} preconditioned with $\mathcal{P}_{SP}$  does not achieve convergence in a single iteration \aw{for all values of $\alpha$ we tested, however small}.

		\begin{table}[!htb]
			
			\centering
			\begin{tabular}{c|c|c}
				$n_x$ & SP ($\alpha = 1)$ & SP ($\alpha = 0.01)$\\
				\hline
				50 & 9	(90, 1602)&	3	(30,	534)\\
				100&9	(90,	2394)	&2	(20,	508)\\
				200& 8	(80,	2128)	&2	(20,	516)\\
				300& 8	(80,	2147)	&2	(20,	500)\\
				400 & 8	(80,	2172) &2	(20,	516)\\
				500 & 7	(70,	1862)&	2	(20, 508)
			\end{tabular}
			\vspace{4mm}
			\caption{Number of iterations for increasing problem dimension when applying $\mathcal{P}_{SP}$ to \eqref{eq:precondsystem} for $\ell = 10$ with $\eta=0.2$. Brackets show the total number of AMG initializations and the total number of matrix--vector products with $A$. \label{tab:nx_scalingAMG}} 
		\end{table}
		
		\begin{table}
			\centering
			\begin{tabular}{c|c|c}
				$\ell$ & SP ($\alpha = 1)$ & SP ($\alpha = 0.01)$\\
				\hline
				6 &8	(48,	1136) &	2	(12,	252)\\
				10& 9	(90,	2398)	&2	(20,	508)\\
				16& 9	(144, 4032)	&2	(32,	896)\\
				20& 10	(200,	5800)	&2	(40,	1144)\\
				30 & 10	(300,	8876)	&2	(60,	1756)\\
				40&10	(400,	11924)	&2	(80,	2376)\\
				50& 11	(550,	16607)	&3	(150,	4467)
			\end{tabular}
			\caption{Number of  iterations for increasing number of blocks, $\ell$, when applying $\mathcal{P}_{SP}$ to \eqref{eq:precondsystem} for $n_x=100$ and $\eta=0.2$. Brackets show the total number of AMG initializations and the total number of matrix--vector products with $A$.  \label{tab:AMGscalingwithl}}
		\end{table}

		Table \ref{tab:nx_scalingAMG} shows the number of outer iterations, AMG initializations, and matrix--vector products required to reach convergence for increasing block size, $\sizeA$. For $\alpha = 1$ we see a decrease in the number of outer iterations with increasing $\sizeA$, similarly to $\mathcal{P}_{NC}$. For $\alpha=0.01$ we obtain dimension-independent convergence in $2$ iterations for $\ell=10$ and $n_x\ge 100$. The number of inner iterations is more stable with increasing problem size than for $\mathcal{P}_{NC1}$. We also note that the inner iteration counts are much smaller than for $\mathcal{P}_{NC2}$. As the iteration counts are constant, the number of initializations of AMG is also constant with increasing problem dimension. Table \ref{tab:AMGscalingwithl} shows how the performance of the system preconditioned with $\mathcal{P}_{SP}$ scales with increasing number of blocks. For both values of $\alpha$, the number of outer iterations increases slightly with $\ell$, but remains no larger than those required for $\mathcal{P}_{NC2}$.  Similarly to $\mathcal{P}_{NC}$, we see an increase in the number of matrix--vector products with ${A}$ and in the number of initializations of AMG as the number of blocks increases.
		
		\begin{figure}
			\centering
			\includegraphics[width=0.6\textwidth]{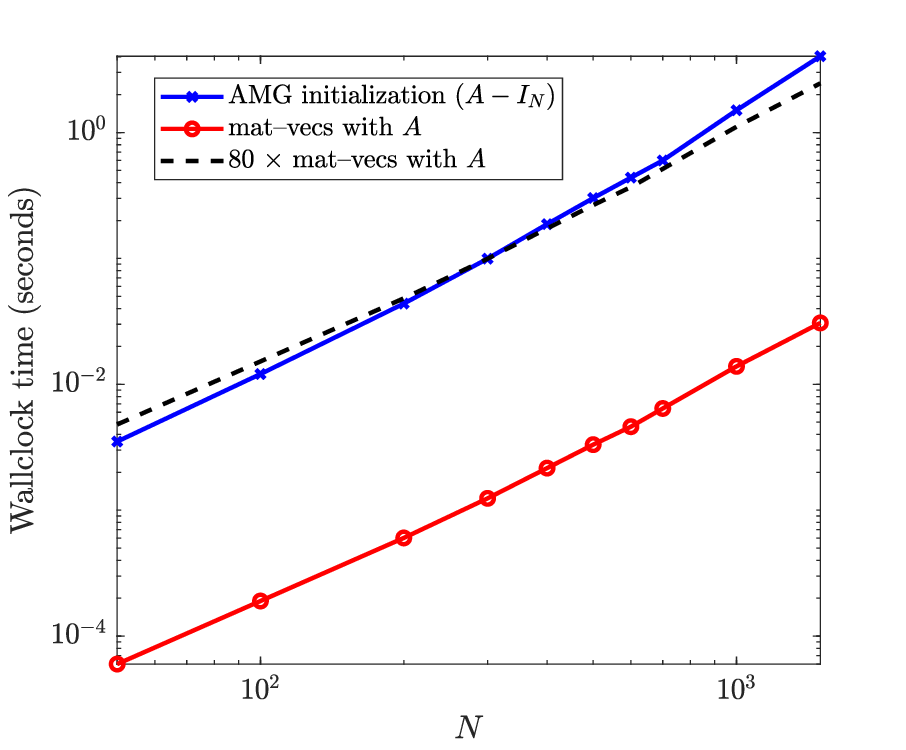}
			\caption{Comparison of wallclock times required for the pre-computation for AMG compared to matrix--vector products for increasing problem dimension. Times are averaged over $50$ realizations.}
			\label{fig:AMGvsMatvec}
		\end{figure}
		In order to study the computational cost of  the preconditioner $\mathcal{P}_{SP}$, we need to account for the cost of initializing AMG. This is likely to be highly dependent on the problem of interest, the implementation of AMG, and the specific computer being used.  
		Figure \ref{fig:AMGvsMatvec} shows how the wallclock time required to initialize the AMG preconditioner for one sub-problem compares to the time needed to compute a matrix--vector product of the same size.  For this experimental framework, the cost of one AMG initialization is approximately the same as 80 matrix--vector products, and this relationship is preserved as we increase the size of the problem being considered. We use this value in the following tests in order to approximately compare the cost of using $\mathcal{P}_{SP}$ to other preconditioners.
		
		\begin{figure}
			\centering
			\includegraphics[width=0.95\textwidth,trim = 20mm 0mm 20mm 0mm, clip]{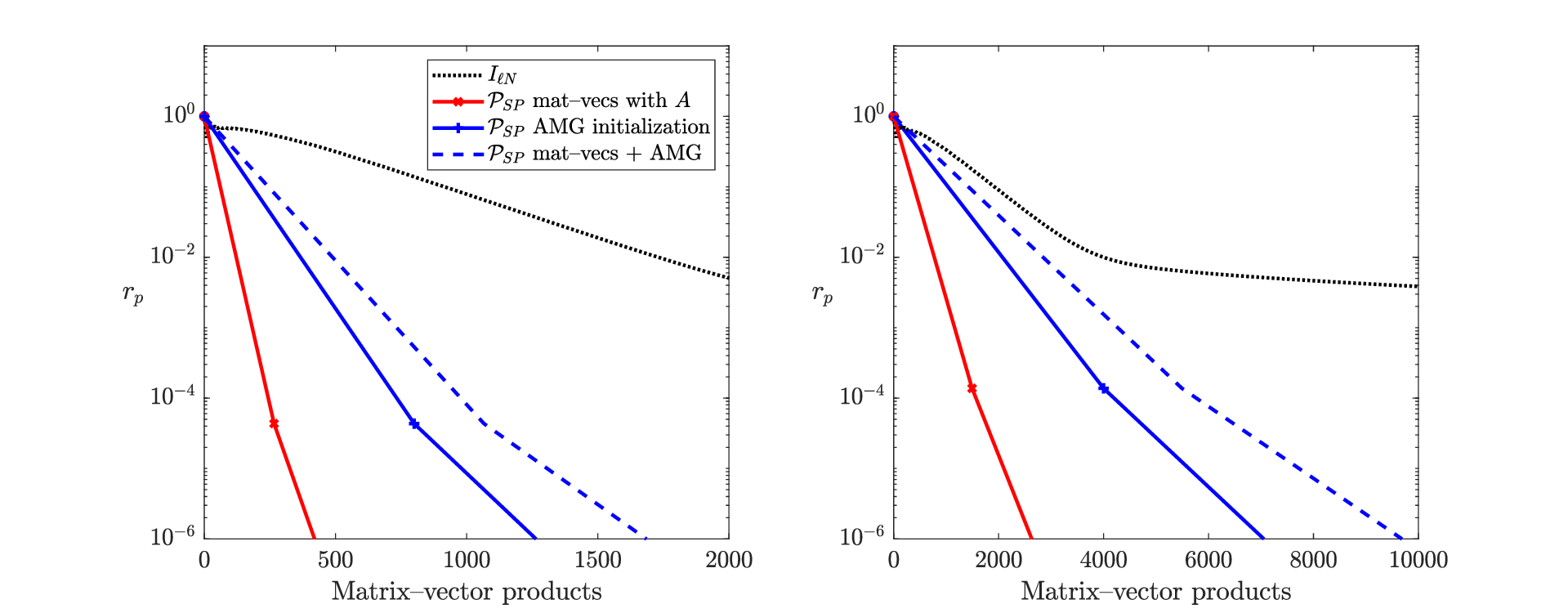}
			\caption{Comparison of the relative residual norm $r_p = \frac{\|b-\mathcal{A}\chi_p\|_2}{\|b\|_2}$ with iterations and matrix--vector products with $A$ for $\ell =10$ (left) and $\ell=50$ (right) using $\mathcal{P}_{SP}$. Matrix--vector products with $A$ are shown in red, with the blue line with crosses showing the number of AMG initializations multiplied by $80$ (in accordance with the results of Figure \ref{fig:AMGvsMatvec}). The blue dashed line shows the total approximate cost in terms of matrix--vector products. Note the difference in the scale of the $x$-axis in on the left and right panels. The full convergence of the unpreconditioned system for this problem is shown in Figure \ref{fig:residual_vs_matvecs}.}
			\label{fig:residual_vs_matvecsSP}
		\end{figure}
		Figure \ref{fig:residual_vs_matvecsSP} shows the decay of the relative residual norm $r_p=\frac{\|b-\mathcal{A}\chi_p\|_2}{\|b\|_2}$ against computational cost in terms of matrix--vector products  when using the preconditioner $\mathcal{P}_{SP}$ for $\ell=10$ (left) and $\ell=50$ (right) when $\alpha = 0.01$. Here we plot the approximate computational cost in terms of matrix--vector products by counting each AMG initialization as 80 matrix--vector products (in accordance with the results of Figure \ref{fig:AMGvsMatvec}).
		We note that for both choices of $\ell$, preconditioning \eqref{eq:precondsystem} with $\mathcal{P}_{SP}$ leads to convergence in two iterations. Even when accounting for the additional cost of AMG initialization, the overall computational cost is much lower than both solving the unpreconditioned problem \eqref{eq:A} or solving \eqref{eq:precondsystem} preconditioned with $\mathcal{P}_{NC2}$ (see Figure \ref{fig:residual_vs_matvecs}). As the number of outer iterations is problem-independent, these performance gains are expected to become even greater with increasing problem size, which we study in the following section.

		\subsection{Comparison of $\mathcal{P}_{NC2}$ and $\mathcal{P}_{SP}$ for high-dimensional problems}

		\begin{table}[]
			
			\centering
			\begin{tabular}{c|cc}
				$n_x$  &  $\mathcal{P}_{NC2}$  &  $\mathcal{P}_{SP}$ \\ \hline
				500  &  7 (--, 7035)& 2 (20, 532)\\
				750  &  7 (--, 10535)& 2 (20, 516)\\
				1000  & 7 (--, 14056)& 2 (20, 516)\\
				1250  & 7 (--, 17535)& 2 (20, 532)\\
				1500  & 6 (--, 18030)& 2 (20, 516) \\
			\end{tabular}
			\caption{Number of outer iterations, AMG initializations, and matrix--vector products with $A$ for increasing problem dimension when applying the preconditioners $\mathcal{P}_{NC2}$ and $\mathcal{P}_{SP}$ for $\ell=10$, $\alpha = 0.01$, and $\resource=0.2$. The dimension of $\mathcal{A}\in\mathbb{R}^{\ell N \times \ell N}$ ranges from $\ell N \in [2.5\times 10^6, 2.25\times 10^7]$.  \label{tab:high_dim}}
			
		\end{table} 
		
		\begin{figure}
			\centering
			\includegraphics[width = 0.6\textwidth]{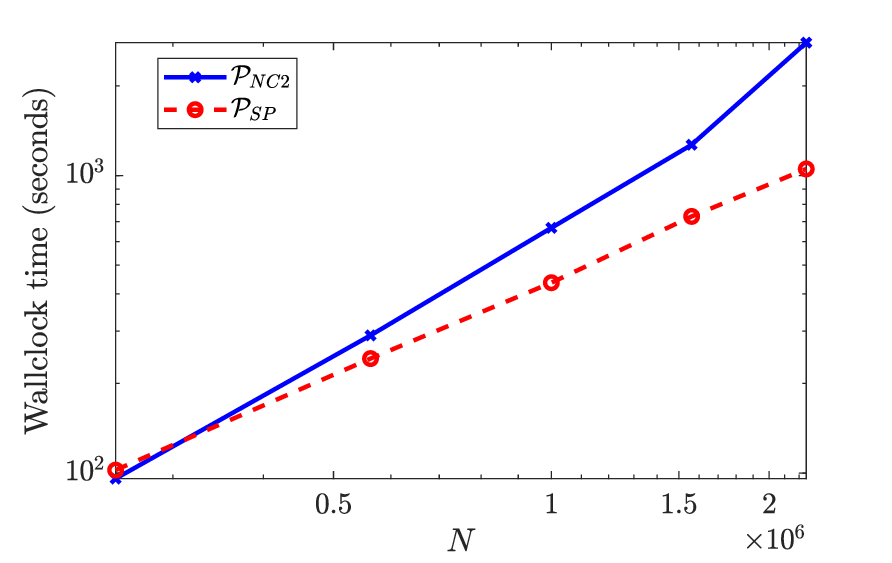}
			\caption{Wallclock times for increasing problem dimension when using the preconditioners $\mathcal{P}_{NC2}$ (red) and $\mathcal{P}_{SP}$ (blue), with $\ell=10$, $\alpha = 0.01$, and $\resource = 0.2$. The dimension of $\mathcal{A}\in\mathbb{R}^{\ell N \times \ell N}$ ranges from $\ell N \in [2.5\times 10^6, 2.25\times 10^7]$.} 
			\label{fig:HighDimTiming}
		\end{figure}
		
		Table \ref{tab:high_dim} shows the performance of the best preconditioners for high-dimensional problems in space: $\mathcal{P}_{NC2}$ and $\mathcal{P}_{SP}$ for $\alpha=0.01$, $\resource = 0.2$, and $\ell=10$. The outer iterations do not increase with problem size for either preconditioner. We note that the number of matrix--vector products required with $A$ is much larger for $\mathcal{P}_{NC2}$ than $\mathcal{P}_{SP}$, and the number of initializations of AMG is kept low due to the small number of outer iterations. Figure \ref{fig:HighDimTiming} shows the wallclock time required to solve the problem with the two preconditioners $\mathcal{P}_{NC2}$ and $\mathcal{P}_{SP}$. For $n_x=500$ (i.e. $\mathcal{A} \in \mathbb{R}^{2,500,000 \times 2,500,000}$), the convergence is faster when using $\mathcal{P}_{NC2}$ than $\mathcal{P}_{SP}$. However, in line with the results of Table \ref{tab:high_dim}, the computational cost of $\mathcal{P}_{NC2}$ scales poorly with increasing problem dimension, and so for very large problems we obtain faster convergence in terms of wallclock time with $\mathcal{P}_{SP}$. We note that these results are implementation-specific, but indicate that if an efficient and reliable AMG implementation can be obtained for the problem of interest, the improved scaling of outer iterations may also yield improved scaling in terms of computational efficiency for large problems.

		\section{Conclusions} \label{sec:Conclusions}
		All-at-once diffusion-based covariance operators are used within ocean data assimilation algorithms, and involve the solution of a linear system with block Toeplitz structure \cite{weaver2018time}. Block $\alpha$-circulant preconditioners are known to be extremely powerful in accelerating the solution of problems of this form. In many cases, implementation of block $\alpha$-circulant preconditioners can be done in a computationally efficient manner, by exploiting fast Fourier transforms to diagonalize the preconditioner.  
		In this paper, we proposed approximations to the block $\alpha$-circulant preconditioner that can  be applied when computing spectral transforms in the spatial dimension is not feasible.
		
		Specifically, we proposed two preconditioning approaches that applied approximations to the block $\alpha$-circulant preconditioner iteratively within an outer Chebyshev semi-iteration (CI). The first approach applied CI to a number of shifted linear systems of the form $A-\lambda_j I_\sizeA$, where $\lambda_j$ is a scaled root of unity. We proved readily computable bounds on the asymptotic convergence factor using only the real part of $\lambda_j$. This permitted the development of a practical extension which ensured {{that}} each sub-problem was solved to approximately the same tolerance, resulting in improved overall convergence. The second approach reformulated a complex shifted linear system to a real-valued saddle point system, which we then solved using preconditioned MINRES.
		
		Numerical results revealed that our approximate preconditioners  are  robust and efficient in terms of outer iterations and matrix--vector products. This performance was independent of the problem dimension for both an increasing block size and an increasing number of blocks, leading to computational gains for high-dimensional problems compared to solving the unpreconditioned system. Unlike the saddle point preconditioned approach, which employs a nonlinear solver (MINRES), the nested Chebyshev approach is completely  linear, which has advantages for representing covariance matrices when the problem is not solved to a high accuracy (see \cite{weaver_correlation_2016,weaver2018time}). For very high-dimensional problems the computational cost of the saddle point preconditioner grew more slowly with the dimension of the blocks than the best nested Chebyshev approach, although this is expected to be strongly dependent on the specific implementation and problem of interest.
		Future work will aim to test the performance of these preconditioners in a realistic data assimilation system such as that of~\cite{Chrust_2025}. 
		
		\jt{We also note that the block $\alpha$-circulant preconditioner can be extended to a problem with similar block structure, with varying diagonal blocks $A_1,A_2,\dots, A_\ell$ using some `representative' or average choice $\widehat{A}$, e.g. following the approach of \cite{Gander2017}. The methods and theory of Sections \ref{sec:ApproxPrecond} and \ref{sec:SP}  apply directly in this setting. However, the eigenvalue bounds on the preconditioned system presented in Section \ref{sec:Background} no longer hold, making the use of this preconditioner within Chebyshev semi-iteration challenging or impossible. }
		
		\section*{Acknowledgements}
		
		JMT and JWP gratefully acknowledge support from the Engineering and Physical Sciences Research Council (EPSRC) UK grant EP/S027785/1.
		
		\appendix
		\section{Eigenspectrum of ${A}$}\label{sec:Append}
		
		Consider the matrix defined in \eqref{IminusLap}. Imposing the Dirichlet boundary conditions for the discretized Laplacian on the unit square, we may use the reasoning of \cite[Lemma 8.1]{Iserles} to deduce that the $n_x^2$ eigenvalues of ${A}$ are given by
		$$
		\mu_{i,j} = 1 + \frac{4 \nu}{h^2} \left( \sin^2 \left ( \frac{i}{2(n_x+1)}\pi \right) + \sin^2 \left ( \frac{j}{2(n_x+1)}\pi  \right ) \right), \quad i, j = 1,2, \ldots, n_x.
		$$
		The minimum and maximum (real and positive) eigenvalues of ${A}$ are
		\begin{align*}
			\mu_{\min} ={}& \mu_{1,1} = 1 + \frac{8 \nu}{h^2} \, \sin^2 \left ( \frac{\pi}{2(n_x+1)} \right )
			> 1, \\
			\mu_{\max} ={}& \mu_{n_x,n_x} = 1 + \frac{8 \nu}{h^2} \, \sin^2 \left ( \frac{n_x\pi}{2(n_x+1)} \right ).
		\end{align*}
		Note that when $n_x$ approaches infinity,
		$\mu_{\max} \approx 1 + \frac{8 \nu}{h^2}$, and $\mu_{\min} \approx 1 + \frac{2 \nu \pi^2}{h^2(n_x +1)^2} = 1 + 2 \nu \pi^2$ since $h = \frac{1}{n_x+1}$.
		
		The associated (normalized) eigenvectors are given by
		$$
		X_{i,j} = V_i \otimes V_j, \quad i, j = 1,2, \ldots, n_x,
		$$
		where
		$$
		V_j = \sqrt{\frac{2}{n_x+1}} \left( \sin\left( \frac{1j\pi}{n_x+1}\right),  \ldots, \sin\left( \frac{n_x j\pi}{n_x+1}\right) \right)^{\top}.
		$$

		\bibliography{bibCirc}
		\bibliographystyle{plain}
	\end{document}